\documentclass[reqno]{amsart}

\usepackage{cite}
\usepackage{bbm}
\usepackage{amsmath,amsfonts,amssymb,amsthm}
\usepackage[utf8x]{inputenc}
\usepackage{paralist}
\usepackage{todonotes}
\usepackage{tikz}
\usetikzlibrary{matrix}
\usetikzlibrary{cd}
\usepackage[pagebackref,colorlinks=true]{hyperref}

\newcommand{\RP}{\mathbbm{RP}}
\newcommand{\R}{\mathbbm{R}}
\newcommand{\D}{\mathbbm{D}}
\newcommand{\C}{\mathbbm{C}}
\newcommand{\K}{{\mathbbm{K}}}
\newcommand{\s}{\mathbbm{S}}
\newcommand{\B}{\mathbbm{B}}
\newcommand{\cL}{\mathcal L}
\newcommand{\e}{\varepsilon}
\newcommand{\0}{\underline 0}
\newcommand{\inpr}[2]{\ensuremath{\langle#1,#2\rangle}}
\newcommand{\norm}[1]{\lVert #1\rVert}  
\newcommand{\set}[2]{\ensuremath{\{\,{#1}\mid {#2}\,\}}}

\DeclareMathOperator{\Cone}{Cone}
\DeclareMathOperator{\Ima}{Im}

\newtheorem{thm}{Theorem}[section]
\newtheorem{lem}[thm]{Lemma}
\newtheorem{prop}[thm]{Proposition}
\newtheorem{cor}[thm]{Corollary}

\theoremstyle{definition}

\newtheorem{examp}[thm]{Example}
\newtheorem{rem}[thm]{Remark}

\numberwithin{equation}{section}

\begin{document}

\title[Equivalence of Milnor and Milnor-L\^e fibrations]{Equivalence of Milnor and Milnor-L\^e fibrations for real analytic maps}


\author[J.~L.~Cisneros-Molina]{Jos\'e Luis Cisneros-Molina}
\address{Instituto de Matem\'aticas, Unidad Cuernavaca\\ Universidad Nacional Aut\'onoma de M\'exico\\ Avenida Universidad s/n, Colonia Lomas
de Chamilpa\\ Cuernavaca, Morelos, Mexico.}
\curraddr{}
\email{jlcisneros@im.unam.mx}
\thanks{The first author is Regular Associate of the Abdus Salam International Centre for Theoretical Physics, Trieste, Italy. Supported by CONACYT~253506.}

\author[A.~Menegon]{Aur\'elio Menegon}
\address{Universidade Federal da Para\'iba \\ Departamento de Matem\'atica \\ CEP 58051-900, Jo\~ao Pessoa - PB, Brazil}
\curraddr{}
\email{aurelio@mat.ufpb.br}



\date{}

\begin{abstract}
In \cite{Milnor:SPCH} Milnor proved that a real analytic  map $f\colon (\R^n,0) \to (\R^p,0)$, where $n \geq p$, with an isolated critical point at the origin has a 
fibration on the tube $f|\colon \B_\e^n \cap f^{-1}(\s_\delta^{p-1}) \to \s_\delta^{p-1}$. Constructing a vector field such that, (1) it is transverse to the spheres, and (2) it is 
transverse to the tubes, he ``inflates'' the tube to the sphere, to get a fibration $\varphi\colon \s_\e^{n-1} \setminus f^{-1}(0) \to \s^{p-1}$, but the projection is not 
necessarily given by $f/ \|f\|$ as in the complex case.

In the case $f$ has isolated critical value, in \cite{Cisneros-Seade-Snoussi:d-regular} it was proved that if the fibres inside a small tube are transverse to the sphere $\s_\e$, then it has a 
fibration on the tube. Also in \cite{Cisneros-Seade-Snoussi:d-regular}, the concept of $d$-regularity was defined, it turns out that $f$ is $d$-regular if and only if the map $f/ \|f\|\colon \s_\e^{n-1} \setminus 
f^{-1}(0) \to \s^{p-1}$ is a fibre bundle equivalent to the one on the tube. 

In this article, we prove the corresponding facts in a more general setting: if a locally surjective map $f$ has a linear 
discriminant $\Delta$ and a fibration on the tube $f|\colon \B_\e^n \cap f^{-1}(\s_\delta^{p-1} \setminus \Delta) \to \s_\delta^{p-1} \setminus \Delta$, then 
$f$ is $d$-regular if and only if the map $f/ \|f\|\colon \s_\e^{n-1} \setminus f^{-1}(\Delta) \to \s^{p-1} \setminus \mathcal{A}$ (with $\mathcal{A}$ the radial projection of $\Delta$ on $\s^{p-1}$) 
is a fibre bundle equivalent to the one on the tube. We do this by constructing a vector field $\tilde{w}$ which inflates the tube to the sphere 
in a controlled way, it satisfies properties analogous to the vector field constructed by Milnor in the complex setting: besides satisfying (1) and (2) above, it also satisfies that $f/ \|f\|$ is 
constant on the integral curves of $\tilde{w}$.

This is a corrected version of the article published in {\em Internat. J. Math.}, 30(14):1950078, 1-25, 2019, where in the proof of Theorem~3.7 (Theorem~\ref{lem:uni.con.str.0} here)
two inequalities were used that do not hold in general. Such inequalities are not essential for the proof and Theorem~3.7 can be proved without them.
\end{abstract}


\maketitle

\dedicatory{\textit{Man approaches the unattainable truth through a succession of errors}.\\ \indent Aldous Huxley}

\section*{Preface}

This is a corrected version of the article \cite{Cisneros-Menegon:EMFMLF} published in {\em Internat. J. Math.}, 30(14):1950078, 1-25, 2019, where in the proof of Theorem~3.7 (Theorem~\ref{lem:uni.con.str.0} here)
two inequalities were used that do not hold in general. Such inequalities are not essential for the proof and Theorem~3.7 can be proved without them. 
Here we give such proof, which just needs the complementary Proposition~\ref{prop:keyprop}. This proposition also simplifies the proof of Lemma~\ref{lem:uni.con.str.1}
since it implies that \textbf{Case~2.B} with $\mu(x)<0$ cannot occur, so the the lengthy construction for that case given in the original article is not necessary.
Hence, all the results in the original article \cite{Cisneros-Menegon:EMFMLF} are valid.

We also take the opportunity to add Proposition~\ref{prop:DF}, Corollary~\ref{prop:formula} and Corollary~\ref{prop:sp.rp}, which give a formula for the differential of the spherification map
and some results about it.
They are not neccesary for the results of the article but can be useful and help to describe in a more precise way the vector fields constructed (see Remark~\ref{rem:mu.value}).
An errata has also being sent to the journal.

\section{Introduction}

Milnor Fibration Theorem is an important result in singularity theory. It is about the topology of the fibres of holomorphic functions near its critical points. To each singular point of a complex
hypersurface it associates a fibre bundle, known as the Milnor Fibration. For an overview of its origin, generalizations and connections with other branches of mathematics we recommend 
the recent survey article by Seade \cite{Seade:OMFTIOA50Y}.

Let $f\colon(\C^n,\0) \to (\C,0)$ be a holomorphic  map  with a critical point at the origin $\0 \in \C^n$. The \textit{Milnor Fibration} \cite[Theorem~4.8]{Milnor:SPCH} is given by
\begin{equation}\label{type 1}
\phi := \frac{f}{|f|}:  \s_\e \setminus K  \longrightarrow\s^1\,,
\end{equation}
where $K$ is the \textit{link} of $f$ at $\0$, that is $K =f^{-1}(0) \cap \s_\e$ with $\s_\e$ being a sufficiently small sphere around $\0$.

There is a second fibre bundle called the \textit{Milnor-L\^e Fibration}; this is a fibre bundle on a \textit{Milnor tube} $N(\e,\delta) = \B_\e \cap
f^{-1}(\partial \D_\delta)$ where $\B_\e$ is the ball bounded by $\s_\e$, $\partial \D_\delta$ is the circle bounding the disc $\D_\delta$
of radius $\delta$ centred at the origin of $\C$. When $0<\delta \ll \e$ the map $f$ induces a fibration
\begin{equation}\label{type 2}
f\colon N(\e,\delta) \rightarrow \partial \D_\delta.
\end{equation}
Its existence was proved by Milnor \cite[Theorem~11.2]{Milnor:SPCH} for the case of isolated singularity and by L\^e in \cite[Theorem~(1.1)]{Le1} for the general case, both using Ehresmann Fibration Theorem
for manifolds with boundary.

Fibrations \eqref{type 1} and \eqref{type 2} are equivalent, this is proved using a smooth vector field on $\B_\e\setminus f^{-1}(0)$ constructed by Milnor \cite[Lemma~5.9]{Milnor:SPCH} 
for an arbitrary holomorphic map $f\colon(\C^n,\0) \to (\C,0)$ with the following properties:
\begin{enumerate}
 \item It is transverse to the spheres centred at the origin contained in $\B_\e$,\label{it:tr.sph}
 \item It is transverse to the Milnor tubes,\label{it:tr.tube}
 \item Given an integral curve $p(t)$ of such vector field $\frac{f(p(t))}{\norm{f(p(t))}}$ is constant.\label{it:cons.arg}
\end{enumerate}
With this vector field the Milnor tube can be ``inflated'' to the sphere, i.\ e., following the flow of the vector field one gets a diffeomorphism between the Milnor tube and the complement on the sphere of a
neighbourhood of the link.

Milnor also proved a Fibration Theorem for real singularities \cite[Theorem~2]{Milnor:ISH} or \cite[Theorem~11.2]{Milnor:SPCH}. Let $n \ge p$ and consider  a real analytic  map 
$f\colon (\R^{n}, \0) \to (\R^p,0)$ with an isolated critical point at $\0$. Let $\B_\e^n$ be a ball of radius $\e>0$ sufficiently small and, as before, let $0<\delta \ll \e$.
First, Milnor proved that there is a fibration on the Milnor tube $N(\e,\delta) = \B_\e^n \cap f^{-1}(\s_\delta^{p-1})$ 
\begin{equation}\label{eq:rMLF}
 f\colon N(\e,\delta) \rightarrow \s_\delta^{p-1},
\end{equation}
where $\s_\delta^{p-1}$ is the sphere of radius $\delta$ centred at the origin of $\R^p$.
Then he proved that on the sphere $\s_\e^{n-1}=\partial\B_\e^n$  one has a fibre bundle
\begin{equation}\label{eq:WMF}
\varphi \colon \s_\e^{n-1} \setminus K \to \s^{p-1}
\end{equation}
where $K = f^{-1}(0) \cap \s_\e^{n-1}$ is the link. This was done constructing a vector field on $\B_\e^n\setminus\{0\}$ such that:
\begin{enumerate}
 \item It is transverse to the spheres centred at the origin contained in $\B_\e^n$,
 \item It is transverse to the Milnor tubes,
\end{enumerate}
and ``inflating'' the Milnor tube to the sphere. Milnor pointed out that his real Fibration Theorem has some weaknesses: the condition that $f$ has an isolated critical point is very restrictive, and
since the vector field to ``inflate'' the tube to the sphere does not satisfy property~\eqref{it:cons.arg} as in the complex case, the projection map $\varphi$ is not necessarily given by $\frac{f}{\norm{f}}$.
In fact, in \cite[page~99]{Milnor:SPCH} Milnor gives an example of a map whose fibration on the sphere does not have projection $\frac{f}{\norm{f}}$.

From these remarks two natural questions arose: 
\begin{enumerate}[(i)]
\item When does the fibration on the sphere have projection given by $\frac{f}{\norm{f}}$?\label{it:SMC}
\item Is it possible to relax the condition that $f$ has isolated critical point and still have fibrations on the tube and on the sphere?\label{it:non-iso}
\end{enumerate}
Question \eqref{it:SMC} for $f$ with isolated critical point was studied by several authors \cite{Jacquemard:These,Jacquemard:FMPAR,Seade:OBDAHVF,Ruas-Seade-Verjovsky:RSMF,RS,DosSantos:UMCSMF}.
In \cite{dosSantos:ERMFQHS} dos Santos considered the following question
\begin{enumerate}[(i)]
\setcounter{enumi}{2}
\item If the fibration on the tube and the fibration on the sphere given by $\frac{f}{\norm{f}}$ both exist, are they equivalent?\label{it:q.equiv}
\end{enumerate}
where he answered it positively for $p=2$ and $f$ quasi-homogeneous.

For question \eqref{it:non-iso}, the natural generalization was to consider real analytic  maps $f\colon (\R^{n}, \0) \to (\R^p,0)$ with an \textit{isolated critical value}, as in the case of holomorphic functions. 
There are several works in this direction \cite{Pichon-Seade:FMAfgb,Cisneros-Seade-Snoussi:d-regular,dosSantos-Tibar:RMGHOBS,Massey:RAMFSLI,dosSantos-etal:SOBSFRM}.
To have the fibration \eqref{eq:rMLF} on the tube in a ball $\B_\e^n$ of radius $\e>0$ sufficiently small, it is necessary that $f$ has the \textit{transversality property}: 
there exists a \textit{solid Milnor tube} $\widehat{N}(\e,\delta)= \B_\e^n \cap f^{-1}(\B_\delta^p)$, with $0<\delta \ll \e$, such that all the fibres in the tube are transverse to $\s_\e$, to be able to apply
Ehresmann Fibration Theorem for manifolds with boundary. Examples of maps with the transversality property are maps with the Thom $a_f$ property \cite[Proposition~5.1, Remark~5.7]{Cisneros-Seade-Snoussi:d-regular}.
Having the fibration on the tube, one can use Milnor's vector field to inflate the tube to the sphere to get a fibration on the sphere, but again, not necessarily with projection $\frac{f}{\norm{f}}$.
The reason why the fibration on the sphere of Milnor's example \cite[page~99]{Milnor:SPCH} does not have projection $\frac{f}{\norm{f}}$ is because the map 
\begin{equation}\label{eq:f/nf}
\phi:=\frac{f}{\norm{f}}\colon \s_\e^{n-1} \setminus K \to \s^{p-1}
\end{equation}
is not a submersion, which is a necessary condition to be a smooth fibre bundle. In \cite[Definition~2.4]{Cisneros-Seade-Snoussi:d-regular} 
the concept of \textit{$d$-regularity} was introduced (see Section~\ref{sec:d-reg} below).
In \cite[Proposition~3.2-(4)]{Cisneros-Seade-Snoussi:d-regular} it was proved that the map $f$ is $d$-regular if and only if the map \eqref{eq:f/nf} is a submersion, 
so $d$-regularity is a necessary condition for \eqref{eq:f/nf} to be a fibre bundle. On the other hand, \cite[Lemma~5.2]{Cisneros-Seade-Snoussi:d-regular} states that $f$ is $d$-regular 
if and only if there exists a vector field on $\B_\e^n\setminus f^{-1}(0)$ which satisfies properties~\eqref{it:tr.sph}, \eqref{it:tr.tube} and \eqref{it:cons.arg} above, which implies that ``inflating'' 
the tube to the sphere we have that \eqref{eq:f/nf} is a fibre bundle which is equivalent to the fibre bundle \eqref{eq:rMLF}, so $d$-regularity is a necessary and sufficient condition for \eqref{eq:f/nf} to be
a fibre bundle. Thus, the existence of such vector field answers questions \eqref{it:SMC} and \eqref{it:q.equiv} for $f\colon (\R^{n}, \0) \to (\R^p,0)$ with isolated critical value, 
satisfying the transversality property.

However, Brodersen \cite{HB14} and Hansen \cite{Hansen:MFTRS} pointed out that the argument of the proof of \cite[Lemma~5.2]{Cisneros-Seade-Snoussi:d-regular} is not complete, 
and Hansen in \cite{Hansen:MFTRS} gives sufficient conditions for the existence of such vector field. 
Since then, other authors \cite{Ribeiro,dosSantos-etal:MHSFEP,dosSantos-Ribeiro:GCEMVF} have also given sufficient conditions. 

In this article we prove of the existence of a vector field satisfying properties~\eqref{it:tr.sph}, 
\eqref{it:tr.tube} and \eqref{it:cons.arg} if and only if $f$ is $d$-regular in the more general setting of real analytic maps with linear discriminant considered in \cite{Cisneros-etal:FTDRDMGNICV}.
Let $f\colon (\B_{\e}^{n}, 0) \to (\R^p,0)$ be a real analytic map with linear discriminant $\Delta$, i.~e., $\Delta$ is a union of line segments 
(see Section~\ref{sec:lin.disc} for the precise definition). In order to simplify notation, suppose that $f$ is locally surjective (see Remark \ref{rem:surjective} below).

Suppose that $f$ satisfies the transversality property, then there exists $\delta>0$ sufficiently small such that the restriction of $f$ to the tube
\begin{equation}\label{eq:tub.disc}
f|\colon \B_\e^n \cap f^{-1}(\s_\delta^{p-1} \setminus \mathcal{A}_\delta) \to (\s_\delta^{p-1} \setminus \mathcal{A}_\delta) 
\end{equation}
is a fibre bundle where $\mathcal{A}_\delta:=\Delta\cap\s^{p-1}_\delta$ (see \cite[Theorem~2.7]{Cisneros-etal:FTDRDMGNICV}). 
 
In \cite[Theorem~3.9]{Cisneros-etal:FTDRDMGNICV} it is also proved that if $f$ is $d$-regular then one has a fibre bundle
\begin{equation}\label{eq:sph.disc}
f/ \|f\|\colon (\s_\e^{n-1} \setminus f^{-1}(\Delta)) \to (\s^{p-1} \setminus \mathcal{A})
\end{equation}
where $\mathcal{A}$ is the radial projection of $\mathcal{A}_\delta$ on $\s^{p-1}$.

The aim of the present article is to give a construction of a vector field with properties~\eqref{it:tr.sph}, \eqref{it:tr.tube} and \eqref{it:cons.arg} which works in any open set of $\R^n$
where the maps $f$ and $f/\Vert f\Vert$ are submersions. Applying this construction to maps $f$ with isolated critical value we get a proof of \cite[Lemma~5.2]{Cisneros-Seade-Snoussi:d-regular}
which follows the idea of the original (incomplete) proof. Also this construction of the vector field allows us to prove the equivalence of fibrations \eqref{eq:tub.disc} and \eqref{eq:sph.disc} for
real analytic maps with arbitrary linear discriminant.

\begin{rem}\label{rem:surjective}
Throughout this paper, we will assume that $f$ is \textit{locally surjective}, that is, the image of $f$ contains an open neighbourhood of the origin in $\R^p$,
and we shall not mention it all the time. Nevertheless, it is easy to see that in the general case the same results 
hold if one intersects the bases of the locally trivial fibrations with their image. This choice is to avoid a heavy notation.
\end{rem}

 \section{The canonical pencil, $d$-regularity and the spherefication}\label{sec:d-reg}
 
Let $U$ be an open neighbourhood of $\0 \in \R^n$, $n >p$, and $f\colon(U,\0) \to (\R^p,0)$ a  non-constant analytic map 
with a critical point at $\0\in\R^n$ and $0\in\R^p$ is an isolated critical value. Equip $V=f^{-1}(0)$ with a Whitney
stratification and let $\B_\e^n$ be an open ball in $\R^n$, centred at $\0$, of sufficiently small radius $\e$, so that every sphere
in this ball, centred at $\0$, meets transversely every stratum of $V$, if $V$ is not an isolated point at the origin. Such a ball
exists by \cite[Corollary~2.9]{Milnor:SPCH} if $f$ has an isolated critical point at $\0$, and by the Bertini-Sard theorem in \cite{Verdier:SWTBS} in general.

We define a family of real analytic spaces as follows.  For each $\ell \in\RP^{p-1}$, consider the line $\mathcal L_{\ell} \subset \R^p$
passing through the origin corresponding to $\ell$ and set $$X_{\ell} = \{x \in U \, \vert \, f(x) \in \cL_\ell\}\,.$$
 If we let $\mathcal L_{\ell}^\perp$ be the hyperplane orthogonal to $\mathcal L_{\ell}$, let
$\varpi_{\ell} \colon \R^p \to \cL_{\ell}^\perp$ be the orthogonal projection, and we set $h_{\ell} = \varpi_{\ell} \circ f \,$, then $X_\ell$ is the vanishing set of $h_\ell$, which is real analytic.
Hence $\{X_\ell\}$ is a family of real analytic varieties parameterized by $\RP^{p-1}$.

Notice that the singular points of $X_\ell$ are contained in $V$. Hence each
$X_\ell \setminus V$ is either a smooth submanifold of $U \subset \R^n$ of
dimension $n - p + 1$ or empty. The union of all the $X_\ell$'s is
all of $U$ and one has $V = \cap X_{\ell} = X_{\ell_1} \cap X_{\ell_2}$ for each pair $\ell_1 \ne \ell_2$.

\begin{rem}\label{rem:gluing.links}
We notice  that each $X_\ell$ is naturally the union of three sets: the points
$x \in U$ such that $f(x)=0$, {\it i.e.}, $x\in V$, and the points
$x \in U$ such that $f(x)$ is in one of the two half lines of
 $ \cL_{\ell} \setminus \{0\}$. Write this as:
\begin{equation}\label{eq:decomposition}
X_{\ell} = E_{\ell}^+ \cup V \cup E_{\ell}^- \,.
\end{equation}
\end{rem}

The family $\{X_\ell\, \vert\, \ell \in \RP^{p-1} \}$ is {\it the canonical pencil} of $f$.
Let $\rho$ be a metric in $\R^n$ induced by some positive definite quadratic form.
The map $f$ is said to be \textit{$d$-regular at $\0$ (with respect to the metric $\rho$)} 
if  there exists $\e > 0$ such that every sphere (for the metric $\rho$) of radius
$\le \e$ centred at $\0$ meets every $X_\ell\setminus V$
transversely whenever the intersection is not empty. If the metric is fixed we just say that $f$ is $d$-regular.

In this article we consider $\rho$ as the Euclidean metric in $\R^n$ and we only consider $d$-regular maps with respect to this metric.

\subsection{The spherefication}
\label{sec:spherefication}

Consider the maps  $\Phi\colon U \setminus V\to\s^{p-1}$ and
$\;\mathfrak{F}\colon U \setminus V \to \R^p\setminus\{0\}$ defined by
 \begin{equation}\label{eq:def.g}
 \Phi(x)=\frac{f(x)}{\norm{f(x)}} \quad \hbox{and} \quad\mathfrak{F}(x)=\norm{x}\Phi(x).
\end{equation}
For each $ \ell \in  \RP^{p-1}$   one has that the line $\cL_\ell$
in $\R^p$ meets the sphere $\s^{p-1}$ in two antipodal points, say
$\pm y_\ell$, and one has:
\begin{equation*}
\Phi^{-1}(y_\ell)=E_\ell^+\quad \text{and}\quad \Phi^{-1}(-y_\ell)=E_{\ell}^-,
\end{equation*}
where $E_\ell^+$ (resp. $E_\ell^-$) is the inverse image under $f$ of one of the two half lines in $\cL_\ell \setminus \{0\}$.

Notice that given $y\in \mathcal{L}_\ell \setminus \{0\} \subset
\R^p\setminus\{0\}$, the fibre $\mathfrak{F}^{-1}(y)$ is the
intersection of $E_\ell^\pm$ with the sphere of radius $\norm{y}$
centred at $\0$. Hence, $\mathfrak{F}$ carries spheres in
$\R^n$ of radius $r>0$ into spheres in $\R^p$ of the same radius,
and each $X_\ell$ is also a union of fibres of $\mathfrak{F}$, just
as it is a union of fibres of $f$. 
The analytic map $\mathfrak{F}$ is called the {\it  spherefication} of $f$.

As we mentioned in the Introduction, a map $f$ is $d$-regular if and only if the map $\frac{f}{\norm{f}}$ is a submersion.
This is given by the following proposition together with other characterizations of $d$-regularity in terms of the spherefication.

\begin{prop}[{\cite[Proposition~3.2]{Cisneros-Seade-Snoussi:d-regular}}]\label{prop:equiv}
Let  $f\colon(\B_\e^n ,\0) \to (\R^p,0)$ be a real analytic map with isolated
critical value at the origin. Set $V=f^{-1}(0)$ and
$K_{\e'}=V\cap\s_{\e'}^{n-1}$. The following are equivalent:
\begin{enumerate}[(A)]
\item The map $f$ is $d$-regular at $0$.\label{it:d-r} \item For
each sphere $\s_{\e'}^{n-1}$ in $\R^n$ centred at $\0$ of radius
${\e'} < \e$, the restriction map $\mathfrak{F}_{\e'}\colon
\s_{\e'}^{n-1} \setminus V\to \s^{p-1}_{{\e'}}$ of $\mathfrak{F}$ is
a submersion.\label{it:res-sub} \item The spherefication map
$\mathfrak{F}$ is a submersion at each $x \in \B_\e^n \setminus
V$.\label{it:sp-sub} \item The map $\phi =
\frac{f}{\norm{f}}\colon \s_{{\e'}}^{n-1} \setminus K_{{\e'}}
\longrightarrow\s^{p-1}$ is a submersion for every sphere
$\s_{{\e'}}^{n-1}$ with ${\e'}< \e$.\label{it:phi-sub}
\end{enumerate}
\end{prop}

The following proposition relates the differential of the
spherefication map $\mathfrak{F}$ and the differential of $f$.

\begin{prop}\label{prop:DF}
Let $x\in U\setminus V$ and $v\in T_x(U\setminus V)$. Then
\begin{equation*}
D\mathfrak{F}_x(v)=\frac{\inpr{x}{v}}{\norm{x}^2}\mathfrak{F}(x)-\frac{f(x)}{\norm{f(x)}^2}\inpr{\mathfrak{F}(x)}{Df_x(v)}+\frac{\norm{x}}{\norm{f(x)}}Df_x(v).
\end{equation*}
\end{prop}

\begin{proof}
 Set
\begin{align*}
x&=(x_1,\dots,x_n)\in \B_\e^n\setminus V, &
v&=(v_1,\dots,v_n)\in T_x(\B_\e^n\setminus V),\\
w&=(w_1,\dots,w_p)=D\mathfrak{F}_x(v), &
u&=(u_1,\dots,u_p)=Df_x(v).
\end{align*}
Using the quotient rule for the derivative we get
\begin{equation}\label{eq:quot.rule}
 \frac{\partial}{\partial x_j}\biggl(\frac{\norm{x}}{\norm{f(x)}}\biggr)=\frac{x_j}{\norm{x}\norm{f(x)}}
-\frac{\norm{x}}{\norm{f(x)}^3}\biggl(\sum_{l=1}^p f_l(x)\frac{\partial f_l(x)}{\partial x_j} \biggr).
\end{equation}
Write $f(x)=(f_1(x),\dots,f_p(x))$; by the definition of the
spherefication map \eqref{eq:def.g} we have that
\begin{equation*}
 \mathfrak{F}(x)=\bigl(\mathfrak{F}_1(x),\dots,\mathfrak{F}_p(x)\bigr)
=\biggl(\frac{\norm{x}}{\norm{f(x)}}f_1(x),\dots,\frac{\norm{x}}{\norm{f(x)}}f_p(x)\biggr).
\end{equation*}
Using \eqref{eq:quot.rule} and the product rule we obtain
\begin{equation*}
 \frac{\partial\mathfrak{F}_i(x)}{\partial x_j}=\frac{x_jf_i(x)}{\norm{x}\norm{f(x)}}
-\frac{\norm{x}}{\norm{f(x)}^3}\biggl(\sum_{l=1}^p f_l(x)\frac{\partial f_l(x)}{\partial x_j} \biggr)f_i(x)
+\frac{\norm{x}}{\norm{f(x)}}\frac{\partial f_i(x)}{\partial x_j}.
\end{equation*}
Recall that
\begin{equation}\label{eq:ui}
 u_i=\sum_{j=1}^{n}\frac{\partial f_i(x)}{\partial x_j}v_j.
\end{equation}
Then using \eqref{eq:ui} we have:
\begin{align}
 w_i&=\sum_{j=1}^{n}\frac{\partial \mathfrak{F}_i(x)}{\partial x_j}v_j\notag\\
&=\sum_{j=1}^{n}\left[\frac{f_i(x)x_jv_j}{\norm{x}\norm{f(x)}}
-\frac{\norm{x}}{\norm{f(x)}^3}\biggl(\sum_{l=1}^p f_l(x)\frac{\partial f_l(x)}{\partial x_j} \biggr)f_i(x)v_j
+\frac{\norm{x}}{\norm{f(x)}}\frac{\partial f_i(x)}{\partial x_j}v_j \right]\notag\\
&=\frac{f_i(x)}{\norm{x}\norm{f(x)}}\sum_{j=1}^{n}x_jv_j
-\frac{\norm{x}}{\norm{f(x)}^3}f_i(x)\sum_{l=1}^p f_l(x)\sum_{j=1}^{n}\frac{\partial f_l(x)}{\partial x_j}v_j
+\frac{\norm{x}}{\norm{f(x)}}u_i\notag\\
&=\frac{\inpr{x}{v}}{\norm{x}\norm{f(x)}}f_i(x)
-\frac{\norm{x}}{\norm{f(x)}^3}f_i(x)\sum_{l=1}^pf_l(x)u_l
+\frac{\norm{x}}{\norm{f(x)}}u_i\notag\\
&=\frac{\inpr{x}{v}}{\norm{x}\norm{f(x)}}f_i(x)
-\frac{\norm{x}}{\norm{f(x)}^3}f_i(x)\inpr{f(x)}{Df_x(v)}
+\frac{\norm{x}}{\norm{f(x)}}u_i\notag\\
&=\frac{\inpr{x}{v}}{\norm{x}\norm{f(x)}}f_i(x)
-\frac{f_i(x)}{\norm{f(x)}^2}\inpr{\mathfrak{F}(x)}{Df_x(v)}
+\frac{\norm{x}}{\norm{f(x)}}u_i\,.\label{eq:DF.coor}
\end{align}
Hence the formula of the proposition follows.
\end{proof}

\begin{cor}\label{prop:formula}
Let $x\in U\setminus V$ and $v\in T_x(U\setminus V)$. Then
\begin{equation*}
\norm{D\mathfrak{F}_x(v)}^2=\frac{\inpr{x}{v}^2}{\norm{x}^2}+\frac{\norm{x}^2 \norm{Df_x(v)}^2-\inpr{\mathfrak{F}(x)}{Df_x(v)}^2}{\norm{f(x)}^2}\,.
\end{equation*}
\end{cor}
\begin{proof}
From \eqref{eq:DF.coor} we have that
\begin{align*}
 w_i^2&=\frac{\inpr{x}{v}^2}{\norm{x}^2\norm{f(x)}^2}f_i(x)^2+\frac{f_i(x)^2}{\norm{f(x)}^4}\inpr{\mathfrak{F}(x)}{Df_x(v)}^2
+\frac{\norm{x}^2}{\norm{f(x)}^2}u_i^2\\
&\qquad-2\frac{\inpr{x}{v}}{\norm{x}\norm{f(x)}^3}\inpr{\mathfrak{F}(x)}{Df_x(v)}f_i(x)^2+2\frac{\inpr{x}{v}}{\norm{f(x)}^2}f_i(x)u_i\\
&\qquad-2\frac{\norm{x}}{\norm{f(x)}^3}\inpr{\mathfrak{F}(x)}{Df_x(v)}f_i(x)u_i.
\end{align*}
Therefore the corollary follows.
\end{proof}

Let $y\in\R^p$ and let $\s^{p-1}_{\norm{y}}$ be the sphere in $\R^p$ with centre at the origin which contains $y$.
Let $\ell_y$ be the line through the origin in $\R^p$ generated by $y$. Hence, the tangent space $T_{y}\R^p$ decomposes as the direct sum of the orthogonal subspaces
\begin{equation*}
T_{y}\R^p=T_{y}\s^{p-1}_{\norm{y}}\oplus \ell_y.
\end{equation*}
Thus, any vector $u\in T_{y}\R^p$ can be written in a unique way as $u=s+r$, with $s\in T_{y}\R^p$ and $r\in \ell_y$. We call $s$ and $r$ respectively,  the \textit{spherical} and \textit{radial parts} of $u$.

\begin{prop}\label{prop:sp.rp}
Let $x\in U\setminus V$, $v\in T_xU$ and consider $D\mathfrak{F}_x(v)$. Then
\begin{enumerate}[I)]
 \item $\frac{\inpr{x}{v}}{\norm{x}^2}\mathfrak{F}(x)$ is the radial part of $D\mathfrak{F}_x(v)$.\label{it:radial}
 \item $-\frac{f(x)}{\norm{f(x)}^2}\inpr{\mathfrak{F}(x)}{Df_x(v)}+\frac{\norm{x}}{\norm{f(x)}}Df_x(v)$ is the spherical part of $D\mathfrak{F}_x(v)$.\label{it:spherical}
\end{enumerate}
\end{prop}

\begin{proof}
Part \ref{it:radial}) is obvious since it is a multiple of $\mathfrak{F}(x)$. For part \ref{it:spherical}) it is enough to see that its inner product with $f(x)$ vanishes.
\begin{multline*}
-\inpr{f(x)}{\frac{f(x)}{\norm{f(x)}^2}\inpr{\mathfrak{F}(x)}{Df_x(v)}}+\inpr{f(x)}{\frac{\norm{x}}{\norm{f(x)}}Df_x(v)}=\\
-\inpr{\mathfrak{F}(x)}{Df_x(v)}+\inpr{\frac{\norm{x}}{\norm{f(x)}}f(x)}{Df_x(v)}=0.
\end{multline*}
\end{proof}

Proposition~\ref{prop:sp.rp} gives a criterium for a vector $v\in T_x\B_\e^n$ to be tangent to the corresponding $E_\ell^\pm$

\begin{cor}\label{cor:v.El}
Let $x\in U\setminus V$ and $v\in T_xU$. Then $v\in T_x E_{\ell_{\mathfrak{F}(x)}}^\pm$ if and only if
\begin{equation*}
 Df_x(v)=\frac{f(x)}{\norm{f(x)}^2}\inpr{f(x)}{Df_x(v)}.
\end{equation*}
\end{cor}

\begin{proof}
We have that $v\in T_x E_{\ell_{\mathfrak{F}(x)}}^\pm$ if and only if the spherical part of $D_x\mathfrak{F}(v)$ is zero. Then the formula follows
from Proposition~\ref{prop:sp.rp}-\ref{it:spherical}.
\end{proof}

\section{Construction of a vector field satisfying \eqref{it:tr.sph}, \eqref{it:tr.tube} and \eqref{it:cons.arg}}

In this section we construct a vector field satisfying \eqref{it:tr.sph}, \eqref{it:tr.tube} and \eqref{it:cons.arg}. First in Subsection~\ref{ssec:lifting} we prove some auxiliary lemmas, 
in particular, a lemma about liftings to $\R^n$ of the gradient vector field of a function $g\colon\R^p\to\R$ via a submersion $f\colon\R^n\to\R^p$. Then, in Subsection~\ref{ssec:construction} 
we apply the lemmas to a real analytic $d$-regular map $f\colon\R^n\to\R^p$ with isolated critical point, to construct the aforementioned vector field.

\subsection{Lifting of gradients}\label{ssec:lifting}

Consider open subsets $U\subset\R^n$ and $U'\subset\R^p$. Let $f\colon U\to U'$ and $g\colon U'\to\R$ be submersions. 
Consider the function $h=g\circ f\colon U\to\R$ which is a submersion, being the composition of two submersions.
Let $\nabla g$ and $\nabla h$ be the gradients of $g$ and $h$ respectively, which are non zero everywhere.
Let $x\in U$ and let $t=h(x)=g(f(x))$, we have that $t$ is a regular value of $g$ and $h$, so $N=g^{-1}(t)$ and $M=f^{-1}(N)=h^{-1}(t)$ are submanifolds of codimension $1$ of $U'$ and $U$ respectively. 
Since $f$ is a submersion, it is transverse to $N$ and therefore we also have $M=f^{-1}(N)=h^{-1}(t)$ and
\begin{equation}\label{eq:tang}
 T_xM=Df_x^{-1}(T_{f(x)}N).
\end{equation}

\begin{lem}\label{lem:gradients}
Let $x\in U$. Then there exists a unique vector $v(x)\in T_xM$ orthogonal to $\ker Df_x$ such that
\begin{equation}\label{eq:goal}
Df_x(v(x)+\alpha(x)\nabla h(x))=\nabla g(f(x)).
\end{equation}
for some positive function $\alpha\colon U\to\R_+$.
\end{lem}

\begin{proof}
The gradient $\nabla g(f(x))$ is normal to $T_{f(x)} N$. Denote by $G$ the line in $T_{f(x)} U'$ generated by $\nabla g(f(x))$. Analogously the gradient $\nabla h(x)$ is normal to $T_xM$. 
Denote by $H$ the line in $T_xU$ generated by $\nabla h(x)$. Hence, we have the following isomorphisms
\begin{equation*}
T_x U \simeq T_x M \oplus H
\end{equation*}
and
\begin{equation*}
T_{f(x)} U' \simeq T_{f(x)} N \oplus G.
\end{equation*}

On the other hand, we have that $\ker Df_x = T_x f^{-1}(f(x)) \subset T_xM$, so there is an isomorphism:
\begin{equation*}
T_x U \cong T_x f^{-1}(f(x)) \oplus \tilde{M} \oplus H
\end{equation*}
where $\tilde{M}$ denotes the orthogonal complement of $T_x f^{-1}(f(x))$ in $T_xM$. Hence $Df_x$ induces an isomorphism:
\begin{equation}\label{eq:iso1}
Df_x|_{ \tilde{M} \oplus H}\colon \tilde{M} \oplus H \to T_{f(x)}N \oplus G.
\end{equation}
By \eqref{eq:tang} $Df_x|_{ \tilde{M}}\colon \tilde{M}\to T_{f(x)}N$ is an isomorphism, therefore $Df_x(\nabla h(x))$ has the form
\begin{equation}
 Df_x(\nabla h(x))=u(x)+\lambda(x) \nabla g(f(x)),
\end{equation}
with $u(x)\in T_{f(x)} N$ and for some $\lambda(x)\neq0$ (otherwise \eqref{eq:iso1} would not be an isomorphism). Set $\alpha(x)=\frac{1}{\lambda(x)}$, by the isomorphism \eqref{eq:iso1} there exist
a unique vector $v(x)$ in $\tilde{M}$ such that $Df_x(v(x))=-\alpha(x) u(x)$, then
\begin{equation*}
Df_x(\alpha(x)\nabla h(x)+v(x))=\frac{1}{\lambda(x)}u(x)+\nabla g(f(x))-\frac{1}{\lambda(x)}u(x)=\nabla g(f(x)).
\end{equation*}

It only remains to check that $\alpha(x)>0$ for every $x\in U$. By the chain rule we have that
\begin{equation}\label{eq:c.r}
 \nabla h(x)=\nabla g(f(x)) \cdot Df_x.
\end{equation}
Taking the inner product with $\nabla g(f(x))$ in \eqref{eq:goal} we have
\begin{align*}
\nabla g(f(x))\cdot Df_x(v(x)+\alpha(x)\nabla h(x))&= \nabla g(f(x))\cdot\nabla g(f(x))\\
\nabla g(f(x))\cdot \bigl(-\alpha(x) u(x)\bigr)+\alpha(x)\nabla g(f(x))\cdot Df_x(\nabla h(x))&= \nabla g(f(x))\cdot\nabla g(f(x))\\
\alpha(x)\nabla h(x)\cdot \nabla h(x)&=\norm{\nabla g(f(x))}^2,\qquad\text{by \eqref{eq:c.r},}\\
\alpha(x)\norm{\nabla h(x)}^2&=\norm{\nabla g(f(x))}^2,
\end{align*}
therefore
\begin{equation}\label{eq:alpha.value}
 \alpha(x)=\frac{\norm{\nabla g(f(x))}^2}{\norm{\nabla h(x)}^2}>0.
\end{equation}
\end{proof}

\begin{cor}\label{cor:lift}
Let $f\colon U\to U'$, $g\colon U'\to\R$ and $h=g\circ f$ be as in Lemma~\ref{lem:gradients}. Let $x\in U$ and let $L_x$ be a subspace of $T_xM$ of codimension $1$ in $T_xM$,  
which is transverse to $T_xf^{-1}(f(x))$ in $T_xM$.
Then there exists a unique vector $\bar{v}(x)\in L_x$ orthogonal to $L_x\cap T_x f^{-1}(f(x))$ such that
\begin{equation}
Df_x(\bar{v}(x)+\alpha(x)\nabla h(x))=\nabla g(f(x)).
\end{equation}
for some positive function $\alpha\colon U\to\R_+$.
\end{cor}

\begin{proof}
We have the isomorphism
\begin{equation}\label{eq:M=F+tM}
T_xM\cong T_x f^{-1}(f(x)) \oplus \tilde{M}
\end{equation}
where $\tilde{M}$ denotes the orthogonal complement of $T_x f^{-1}(f(x))$ in $T_xM$. Consider the projection
\begin{equation}
\pi\colon T_x f^{-1}(f(x)) \oplus \tilde{M}\to \tilde{M},
\end{equation}
and its restriction to $L_x$
\begin{equation}\label{eq:p.L}
\pi|_{L_x}\colon L_x\to \tilde{M}. 
\end{equation}
We claim that \eqref{eq:p.L} is onto. Recall that $\dim T_xM=n-1$, and $\dim L_x=n-2$, since $L_x$ is a codimension $1$ subspace of $T_xM$. The kernel of \eqref{eq:p.L} is
$L_x\cap T_x f^{-1}(f(x))$, which has dimension $n-p-1$, since $\dim T_x f^{-1}(f(x))=n-p$ and by hypothesis $L_x$ and $T_x f^{-1}(f(x))$ intersect transversely in $T_xM$. Thus, by the
rank-nullity theorem, the rank of \eqref{eq:p.L} is $n-2-(n-p-1)=p-1$, and since $\dim\tilde{M}=p-1$ \eqref{eq:p.L} is onto as claimed. Since $\ker\pi|_{L_x}=L_x\cap T_x f^{-1}(f(x))$
the orthogonal complement of $L_x\cap T_x f^{-1}(f(x))$ in $L_x$ is sent isomorphically onto $\tilde{M}$ by $\pi|_{L_x}$.

Let $v(x)\in\tilde{M}$ be the vector given by Lemma~\ref{lem:gradients}, let $\bar{v}(x)\in L_x$ the unique vector orthogonal to $L_x\cap T_x f^{-1}(f(x))$ such that $\pi(\bar{v}(x))=v(x)$. 
By isomorphism \eqref{eq:M=F+tM} the vector
$\bar{v}(x)$ has the form
\begin{equation}
\bar{v}(x)=k(x)+v(x),\quad\text{with $k(x)\in T_x f^{-1}(f(x))$.}
\end{equation}
Therefore
\begin{equation*}
Df_x(\bar{v}(x)+\alpha(x)\nabla h(x))=Df_x(v(x)+\alpha(x)\nabla h(x))=\nabla g(f(x)).
\end{equation*}
\end{proof}

\begin{lem}\label{lem:nod1}
Let $r,s\in\R^n$ such that they do not point in opposite directions. Let $v_1,v_2\in\R^n$ be orthogonal vectors to both $r$ and $s$. Then the vectors
$v_1+\alpha r$ and $v_2+\beta s$, with $\alpha,\beta\in\R^+$, cannot point in opposite directions.
\end{lem}

\begin{proof}
Suppose $v_1+\alpha r$ and $v_2+\beta s$, with $\alpha,\beta\in\R^+$ point in opposite directions,
then there exists $0<\lambda\in\R$ such that
\begin{equation*}
v_1+\alpha r=-\lambda(v_2+\beta s).
\end{equation*}
By hypothesis, the subspace generated by $r$ and $s$ is orthogonal to the subspace generated by $v_1$ and $v_2$, so there is no contribution by $r$ and $s$ on the subspace generated by $v_1$ and $v_2$ and
we need to have $v_1=-\lambda v_2$, which implies that $\alpha r=-\lambda\beta s$, which contradicts the fact that $r$ and $s$ do not point in opposite directions.
\end{proof}

\begin{lem}\label{lem:nod2}
Suppose $r,s\in\R^n$ point in opposite directions. Let $0\neq k\in\R^n$ be orthogonal to $r$ (and therefore also to $s$). Then the vectors $r+k$ and $s$ cannot point in opposite directions. 
\end{lem}

\begin{proof}
Since $r$ and $s$ point in opposite directions there exists $0<\lambda\in\R$ such that $r=-\lambda s$. Suppose $r+k$ and $s$ point in opposite directions, then there exists $0<\mu\in\R$ such that
$r+k=-\mu s$, then $-\lambda s+k=-\mu s$, hence $k=(\lambda-\mu) s$. But $k$ and $s$ are orthogonal, so we need to have $\lambda=\mu$ and $k=0$ which is a contradiction.
\end{proof}

\setcounter{subsection}{1}\setcounter{equation}{10}
\subsection{Construction of the vector field}\label{ssec:construction}

Let $U$ be an open neighbourhood of $\0 \in \R^n$. Let  $f\colon(U,0) \to (\R^p,0)$ be a real analytic  map, with isolated critical value at $0\in\R^p$. Also suppose $f$ is $d$-regular and set $V=f^{-1}(0)$.
Let $g\colon\R^p\to\R$ be given by $g(y)=\norm{y}^2$. We have that the restriction 
$f\colon U\setminus V\to\R^p \setminus \{0\}$ is a submersion. On the other hand, since
$f$ is $d$-regular, by Proposition~2.2(C) the spherefication $\mathfrak{F} \colon U\setminus V\to\R^p \setminus \{0\}$ of $f$ is also a submersion. 
Define the maps 
\begin{align*}
h&=g\circ f\colon U\setminus V\to\R,\\
\mathfrak{H}&=g\circ\mathfrak{F} \colon U\setminus V\to\R.
\end{align*}
Notice that by the definitions
\begin{equation}\label{eq:hs}
\begin{aligned}
h(x)&=\norm{f(x)}^2,\\
\mathfrak{H}(x)&=\norm{\mathfrak{F}(x)}^2=\left\Vert\norm{x}\frac{f(x)}{\norm{f(x)}}\right\Vert^2=\norm{x}^2\frac{\norm{f(x)}^2}{\norm{f(x)}^2}=\norm{x}^2.
\end{aligned}
\end{equation}

Let $\B_\e^n$ be an open ball in $U$, centred at $\0$, of sufficiently small radius $\e$ as in Section~2. 
Let $x\in \B_{\e}^n\setminus V$ and let $\delta^2=h(x)$, then $h^{-1}(\delta^2)=\B_{\e}^n \cap f^{-1}(\s_\delta^{k-1} \setminus \{0\})$, that is, the fibres of $h$ are the tubes. On the other hand,
let $\nu^2=\mathfrak{H}(x)$, then $\mathfrak{H}^{-1}(\nu^2)=\s_{\nu}^{n-1}\setminus V$, that is, the fibres of $\mathfrak{H}$ are the spheres.

Consider the corresponding gradient vector fields $\nabla h(x)$ and $\nabla \mathfrak{H}(x)=2x$ 
which are non-zero for any $x\in U\setminus V$. The vector field $\nabla h$ is normal to the Milnor tubes, while the vector field $\nabla \mathfrak{H}$ is normal to the spheres.
The gradient $\nabla g$ is the radial vector field $\nabla g(y)=2y$ on $\R^p$. 

\begin{prop}\label{prop:NOD}
There exists $0<\e'\leq\e$ such that for $x\in \B_{\e'}^n\setminus V$ the vector fields $\nabla h(x)$ and $\nabla \mathfrak{H}(x)$ cannot point in exactly opposite directions at any point.
\end{prop}

\begin{proof}
Apply \cite[Corollary~3.4]{Milnor:SPCH} to the non-negative analytic functions $h$ and $\mathfrak{H}$ but using the Analytic Curve Selection 
Lemma (see Proposition~2.2 of \cite{Burghelea-Verona:LHPAS} or (2.1) of \cite{Looijenga:ISPCI}).
\end{proof}

Let  $f\colon(\B_{\e}^n,0) \to (\R^p,0)$ be a real analytic  map, with isolated critical value at $0\in\R^p$. 
Also assume that $\e>0$ is small enough such that Proposition~\ref{prop:NOD} holds.
Let $\{X_\ell\}_{\ell\in\RP^{p-1}}$ be the canonical pencil of $f$ and consider the decomposition $X_\ell= E_{\ell}^+ \cup V \cup E_{\ell}^-$ of each of its elements given in \eqref{eq:decomposition}.
From now on, to simplify notation we will write $E_\ell$ instead of $E_\ell^\pm.$


Let $x  \in \mathring{\B}_{\e}^n \setminus V$, let 
\begin{align}
w_f(x)&=v_f(x)+\alpha(x)\nabla h(x),\quad\text{and}\label{eq:wf}\\
w_{\mathfrak{F}}(x)&=v_{\mathfrak{F}}(x)+\beta(x)\nabla\mathfrak{H}(x)=v_{\mathfrak{F}}(x)+\nabla\mathfrak{H}(x),\label{eq:wF}
\end{align}
be, respectively, the normal liftings of the radial vector field $\nabla g$  obtained from Lemma~\ref{lem:gradients}  applied to $h$ and $\mathfrak{H}$, 
with $\inpr{v_f(x)}{\nabla h(x)}=0$, $\inpr{v_{\mathfrak{F}}(x)}{\nabla\mathfrak{H}(x)}=0$, $\alpha(x)>0$ and $\beta(x)>0$. We have that
$\beta(x)=1$, by equation \eqref{eq:alpha.value} and the definitions of $\mathfrak{F}$ and $\mathfrak{H}$ (or see equation \eqref{eq:beta} below).
Since both are liftings of $\nabla g$, both are tangent to the corresponding $E_\ell$; $w_f(x)$ is normal to the fibres of $f$ and  to the Milnor tubes, while $w_{\mathfrak{F}}(x)$ is normal to the fibres of $\mathfrak{F}$ and to the spheres.

Let us analyse how are $w_f(x)$ and $w_{\mathfrak{F}}(x)$ on the set of points where the tangent space of the fibre of $f$ coincides with the tangent space of the fibre of $\mathfrak{F}$. 

Let $x\in\mathring{\B}^n_{\e} \setminus V$, denote by $\ell_x$ the ray from $0\in\R^p$ passing by $f(x)$. Let $P_x$ be the normal space of $E_{\ell_x}$ at $x$,
the $p$-dimensional subspace $\langle P_x,\nabla h(x)\rangle$ generated by $P_x$ and $\nabla h(x)$ is the normal space to the fibre $f^{-1}(f(x))$ at $x$.
If the vector $\nabla\mathfrak{H}(x')$ is in $\langle P_x,\nabla h(x)\rangle$,  then the tangent spaces to the fibres of $f$ and $\mathfrak{F}$ at $x$ coincide.
Let $M(f)$ be the set defined by 
\begin{align*}
M(f)&=\set{x\in\mathring{\B}^n_{\e} \setminus V}{\nabla\mathfrak{H}(x)\in\langle P_x,\nabla h(x)\rangle},\\
&=\set{x\in\mathring{\B}^n_{\e} \setminus V}{T_xf^{-1}(f(x))=T_x\mathfrak{F}^{-1}(\mathfrak{F}(x))}.
\end{align*}
Let $F\colon U\to\R^{p+1}$ be the map defined by $F(x)=(f(x),\norm{x}^2)$. We have that $M(f)$ is the set of critical points of $F$.
It is a semi-analytic set with semi-analytic components each of which having $0$ in its closure if $\e$ is sufficiently small.\footnote{We learned this description from \cite{HB14}.}

Let $x\in M(f)$, the tangent space $\mathbf{T}_x=T_xf^{-1}(f(x))=T_x\mathfrak{F}^{-1}(\mathfrak{F}(x))$ is a subspace of $T_x E_{\ell_x}$ with codimension $1$.
Since $w_f(x),w_\mathfrak{F}(x)\in T_x E_{\ell_x}$ and both of them are orthogonal to $\mathbf{T}_x$, they are collinear, that is,
\begin{align}
 w_\mathfrak{F}(x)&=\mu(x)w_f(x),\qquad\text{with $\mu(x)\neq0$.}\label{eq:mu.defi}
\intertext{Hence we have}
 w_f(x)&=\frac{1}{\mu(x)}w_\mathfrak{F}(x).\label{eq:mu.inve}
\end{align}
Restricting $x$ to a component of $M(f)$ we have that either $\mu(x)>0$ or $\mu(x)<0$ for every $x$ in such component.

By \eqref{eq:wf}, \eqref{eq:mu.defi}, \eqref{eq:wF} and \eqref{eq:mu.inve} we have that
\begin{align}
\inpr{w_\mathfrak{F}(x)}{\nabla h(x)}&=\alpha(x)\mu(x)\norm{\nabla h(x)}^2\label{eq:pi.wF.h}\\
\inpr{w_f(x)}{\nabla \mathfrak{H}(x)}&=\frac{1}{\mu(x)}\norm{\nabla \mathfrak{H}(x)}^2.\label{eq:pi.wf.H}
\end{align}

\begin{prop}\label{prop:keyprop}
Let $x\in M(f)$ and let $w_f(x)$ and $w_\mathfrak{F}(x)$ be the normal liftings of the radial vector field $\nabla g$ obtained by applying Lemma~\ref{lem:gradients} to $h$ and $\mathfrak{H}$ respectively. Then $w_f(x)$ and $w_\mathfrak{F}(x)$ cannot point in opposite directions, that is, $\mu(x)>0$, for $\e$ sufficiently small.
\end{prop}

\begin{proof}
Let $x\in M(f)$ and let $\delta^2=\norm{f(x)}^2=h(x)$. Consider the \textit{open solid Milnor tube}
\begin{equation*}
\mathring{N}(\e,\delta):= \B_\e^n \cap f^{-1}(\mathring{\B}_\delta^p).
\end{equation*}
Consider the inverse image of the open interval $(-\delta^2,\delta^2)$ under the map $h$. We have that
\begin{equation*}
h^{-1}((-\delta^2,\delta^2))=h^{-1}([0,\delta^2))=\mathring{N}(\e,\delta)
\end{equation*}
and $x$ is in the closure of $\mathring{N}(\e,\delta)$.
Let $\mathring{\B}^n_{\norm{x}}$ be the open ball centered at $0$ of radius $\norm{x}$ and let $W$ be the  open set given by the intersection
\begin{equation*}
 W=\mathring{N}(\e,\delta)\cap\mathring{\B}^n_{\norm{x}},
\end{equation*}
which is nonempty since $\0\in W$.
By definition we have that 
\begin{equation}\label{eq:W.inteval}
 h(W)\subset[0,\delta^2).
\end{equation}
If $\norm{x}$ is sufficiently small we have that $x$ is in the closure of $W$.

Suppose $w_f(x)$ and $w_\mathfrak{F}(x)$ point in opposite directions, that is, $\mu(x)<0$ in \eqref{eq:mu.defi}, then by \eqref{eq:pi.wf.H} we have that
$\inpr{w_f(x)}{\nabla \mathfrak{H}(x)}<0$.

Since $w_f(x)\in T_x E_{\ell_x}$ there exist a curve $\beta\colon(-\eta,\eta)\to E_{\ell_x}$ such that $\beta(0)=x$ and $\beta'(0)=w_f(x)$. By \eqref{eq:wf} and \eqref{eq:pi.wf.H} we have that
\begin{align}
\frac{d}{dt}h(\beta(t))|_{t=0}&=\inpr{w_f(x)}{\nabla h(x)}=\alpha(x)\norm{\nabla h(x)}^2>0,\label{eq:incre}\\
\frac{d}{dt}\mathfrak{H}(\beta(t))|_{t=0}&=\inpr{w_f(x)}{\nabla \mathfrak{H}(x)}=\frac{1}{\mu(x)}\norm{\nabla \mathfrak{H}(x)}^2<0.\label{eq:decre}
\end{align}
From \eqref{eq:incre} we have that $h(\beta(t))=\norm{f(\beta(t))}^2$ is an increasing function, thus we have that
\begin{equation}\label{eq:f.incre}
\norm{f(\beta(t))}^2>\norm{f(\beta(0))}^2=\norm{f(x)}^2=\delta^2,\quad\text{for $t>0$.}
\end{equation}
From \eqref{eq:decre} we have that $\mathfrak{H}(\beta(t))=\norm{\mathfrak{F}(\beta(t))}^2=\norm{\beta(t)}^2$ is a decreasing function, thus we have
\begin{equation}
\norm{\beta(t)}^2<\norm{\beta(0)}^2=\norm{x}^2,\quad\text{for $t>0$.}
\end{equation}
The vector $w_f(x)$ is transverse to the Milnor tube $N(\e,\delta)=h^{-1}(\delta^2)$ which is the boundary of the open Milnor tube $\mathring{N}(\e,\delta)$
and by \eqref{eq:wF} 
\begin{equation*}
\inpr{w_\mathfrak{F}(x)}{2x}=\inpr{w_\mathfrak{F}(x)}{\nabla \mathfrak{H}(x)}=\norm{\nabla \mathfrak{H}(x)}^2>0,
\end{equation*}
so $w_\mathfrak{F}(x)$ is transverse to the sphere $\s_{\norm{x}}^{n-1}$ pointing outwards. 
By hypothesis, $w_f(x)$ points in the opposite direction of $w_\mathfrak{F}(x)$, hence, $w_f(x)$ is
transverse to the sphere $\s_{\norm{x}}^{n-1}$ pointing \textit{inwards}. Since $x$ is in the closure of $W$, we have that
\begin{equation}
 \beta(t)\in W,\quad\text{for $t>0$,}
\end{equation}
but then \eqref{eq:f.incre} and \eqref{eq:W.inteval} give a contradiction with the continuity of the map $h$. Therefore, $w_f(x)$ and $w_\mathfrak{F}(x)$ cannot point in opposite directions.
Since $x$ is an arbitrary point in $M(f)$, $w_f(x)$ and $w_\mathfrak{F}(x)$ must point in the same direction in all the components of $M(f)$ in a sufficiently small sphere $\B_\e^n$. 
\end{proof}


\begin{lem}\label{lem:uni.con.str.1}
Let $g\colon\R^p\to\R$ given by $g(y)=\norm{y}^2$. If the map $f$ is $d$-regular, then for every $x\in\mathring{\B}_{\e}^n\setminus V$ there is a neighbourhood $U_x$ of $x$ and 
vector fields $w_f$ and $w_\mathfrak{F}$ in $U_x$, which are differentiable liftings of the radial vector field $\nabla g$  
by $f$ and its spherefication $\mathfrak{F}$, respectively, such that:

\begin{enumerate}[(a)]\setlength{\itemsep}{0pt}
\item $w_f$ is tangent to each $E_\ell$, it is transverse to the fibres of $f$ and therefore, it is transverse to the tubes.\label{it:wf}
\item $w_\mathfrak{F}$ is tangent to each $E_\ell$, it is transverse to the fibres of $\mathfrak{F}$, and therefore to the spheres.\label{it:wF}
\item  At every point  $x'\in U_x$ the vectors $w_f(x')$ and $w_\mathfrak{F}(x')$ do not point in opposite directions.\label{it:wf.wF.nod}
\end{enumerate} 
 \end{lem}

\begin{proof}
Recall that by Proposition~\ref{prop:equiv}-(\ref{it:sp-sub}) the spherefication $\mathfrak{F}$ is also a submersion in $\mathring{\B}_{\e}^n \setminus V$. Set $h =g \circ f$ and $ \mathfrak{H} =g\circ\mathfrak{F}$
as before.

\vskip.2cm

\paragraph{\textbf{Case 1}} \textit{Let $x  \in \mathring{\B}_{\e}^n \setminus V$ be such that the vectors $\nabla h(x)$ and $\nabla\mathfrak{H}(x)$ are colinear}. Let 
$w_f(x)$ and $w_\mathfrak{F}(x)$ be the normal liftings of the radial vector field $\nabla g$ obtained by applying Lemma~\ref{lem:gradients} to $h$ and $\mathfrak{H}$ respectively.
By Proposition~\ref{prop:NOD}, $\nabla h(x)$ and $\nabla\mathfrak{H}(x)$ point in the same direction and $v_f(x)$ and $v_{\mathfrak{F}}(x)$ are orthogonal to them. Thus, by Lemma~\ref{lem:nod1}, $w_f(x)$ and
$w_{\mathfrak{F}}(x)$ cannot point in opposite directions, and the lemma is proved in this case.

\vskip.2cm

\paragraph{\textbf{Case 2}} \textit{The vectors $\nabla h(x)$ and $\nabla\mathfrak{H}(x)$ are linearly independent.} In this case the tube 
$N_x=\B_{\e}^n \cap f^{-1}(\s_{\norm{f(x)}}^{k-1})$ and the sphere $S_x=\s_{\norm{x}}^{n-1}$,  are transverse at $x$. Hence, its intersection
$N_x\cap S_x$ is a submanifold of codimension $2$ of $U\subset\R^n$. Suppose $x\in E_{\ell}$ and, as before, denote by $P_x$ the normal space of $E_\ell$ at $x$, so we have that $\dim P_x=p-1$. The tube $N_x$ and
$E_\ell$ are transverse at $x$ since $f$ is a submersion, hence $\nabla h(x)$ is not in $P_x$ and the $p$-dimensional subspace $\langle P_x,\nabla h(x)\rangle$ generated by $P_x$ and $\nabla h(x)$ is
the normal space to the fibre $f^{-1}(f(x))$ at $x$. Since $f$ is $d$-regular, $E_\ell$ and $S_x$ are transverse, so $\nabla\mathfrak{H}(x)$ is not in $P_x$ and the $p$-dimensional subspace 
$\langle P_x,\nabla\mathfrak{h}(x)\rangle$ generated by $P_x$ and $\nabla\mathfrak{h}(x)$ is the normal space to the fibre $\mathfrak{F}^{-1}(\mathfrak{F}(x))$.

\vskip.2cm   

We have two subcases.

\paragraph{\textbf{Subcase A}} \textit{The vector $\nabla\mathfrak{H}(x)$ is not in  $\langle P_x,\nabla h(x)\rangle$}. Then $N_x\cap S_x$ is transverse to the fibre
$f^{-1}(f(x))$, {\it i.e.},  $T_x(N_x\cap S_x)$ is transverse to  $T_xf^{-1}(f(x))$. By Corollary~\ref{cor:lift} there exists a lifting of $\nabla g(f(x))$ by $f$,  
\begin{equation*}
w_f(x)=\bar{v}_f(x)+\alpha(x)\nabla h(x) \,,
\end{equation*}
 with $\bar{v}_f(x)\in T_x(N_x\cap S_x)$. Notice that the hypothesis $\nabla\mathfrak{H}\notin\langle P_x,\nabla h(x)\rangle$ implies
$\nabla h(x)\notin\langle P_x,\nabla\mathfrak{H}(x)\rangle$, so $N_x\cap S_x$ is also transverse to the fibre $\mathfrak{F}^{-1}(\mathfrak{F}(x))$. 
Hence  $T_x(N_x\cap S_x)$ is transverse to  $T_x\mathfrak{F}^{-1}(\mathfrak{F}(x))$. By Corollary~\ref{cor:lift} there exists a lifting of $\nabla g(\mathfrak{F}(x))$ by $\mathfrak{F}$, 
\begin{equation*}
w_\mathfrak{F}(x)=\bar{v}_\mathfrak{F}(x)+\beta(x)\nabla \mathfrak{H}(x) \,,
\end{equation*}
 with $\bar{v}_\mathfrak{F}(x)\in T_x(N_x\cap S_x)$. Since $\bar{v}_f(x)$ and $\bar{v}_\mathfrak{F}(x)$ are in $T_x(N_x\cap S_x)$, they are orthogonal 
to $\nabla h(x)$ and $\nabla \mathfrak{H}(x)$, which do not point in opposite directions. Then Lemma~\ref{lem:nod1} implies that  $w_f(x)$ and $w_\mathfrak{F}(x)$ do not point in opposite directions, 
and the lemma is proved in this case.

\vskip.2cm   

\paragraph{\textbf{Subcase B}} \textit{The vector $\nabla\mathfrak{H}(x')$ is in $\langle P_x,\nabla h(x)\rangle$}. Then $x\in M(f)$.
Consider the normal liftings $w_f(x)$ and $w_\mathfrak{F}(x)$ of the radial vector field $\nabla g$  obtained by applying Lemma~\ref{lem:gradients} to $h$ and $\mathfrak{H}$ respectively.
We have that $w_f(x)$ and $w_\mathfrak{F}(x)$ are collinear and by Proposition~\ref{prop:keyprop}, if $\e$ is small enough, they point in the same direction. 
Then there exists a neighbourhood $U_x$ of $x$ such that $w'_f(x')$ and $w_{\mathfrak{F}}(x')$ do not point in opposite directions for any $x'\in U_x$ and we are done.
\end{proof}

\begin{thm}[{\cite[Lemma~5.2]{Cisneros-Seade-Snoussi:d-regular}}]\label{lem:uni.con.str.0}  
The map $f$ is $d$-regular, if and only if there exists  an analytic vector field $\tilde{w}$ on
 $\mathring{\B}_{\e} \setminus V$ which has the following properties:
\begin{enumerate}[(1)]\setlength{\itemsep}{0pt}
\item It is radial, {\it i.e.}, it is transverse to all spheres in $\mathring{\B}_{\e}$ centred at $0$.\label{it:pr1}
\item It is transverse to all the tubes $f^{-1}(\s_\delta^{k-1})$.\label{it:pr3}
\item It is tangent to each $E_\ell$, whenever it is not empty.\label{it:pr2}
\end{enumerate}
\end{thm}

\begin{proof} It is clear that  properties (\ref{it:pr1}) and (\ref{it:pr2}) imply $d$-regularity. For the proof that $d$-regularity implies the existence of a vector field as stated, 
as in \cite[Lemma~5.9]{Milnor:SPCH}, it suffices to construct such  vector field locally, and then use a partition of unity to get the desired global vector field.

For every $x\in\mathring{\B}^n_{\e} \setminus V$ let $U_x$ be the neighbourhood of $x$ and $w_f$ and $w_\mathfrak{F}$ the vector fields in $U_x$ given by Lemma~\ref{lem:uni.con.str.1}.
Remember that in all cases the vectors $w_f(x')$ and $w_\mathfrak{F}(x')$ have the form

\begin{align}
w_f(x')&=\hat{v}_f(x')+\alpha(x')\nabla h(x'),\label{eq:forma}\\
w_{\mathfrak{F}}(x')&=\hat{v}_{\mathfrak{F}}(x')+\nabla\mathfrak{H}(x'),\label{eq:form}
\end{align}
with $\inpr{\hat{v}_f(x')}{\nabla h(x')}=0$, $\inpr{\hat{v}_{\mathfrak{F}}(x')}{\nabla\mathfrak{H}(x')}=0$, $\alpha(x')>0$ and $\beta(x')>0$ for all $x'\in U_x$. 

Now on each $U_x$ let $\tilde{w}$ be the vector field given by
\begin{equation}\label{eq:the.vf}
\tilde{w}(x')=\norm{\nabla\mathfrak{H}(x')} w_f(x')+\alpha(x')\norm{\nabla h(x')} w_\mathfrak{F}(x').
\end{equation}
Let us check that it has the desired properties. Since $w_f(x')$ and $w_\mathfrak{F}(x')$ are both tangent to $E_\ell$, then $\tilde{w}(x')$ is also tangent to $E_\ell$,
so it satisfies property \eqref{it:pr2}.

By \eqref{eq:the.vf} and \eqref{eq:form} we have that
\begin{align}
\inpr{\tilde{w}&(x')}{\nabla h(x')}=\alpha(x')\norm{\nabla \mathfrak{H}(x')}\norm{\nabla h(x')}^2+\alpha(x')\norm{\nabla h(x')}\inpr{w_{\mathfrak{F}}(x')}{\nabla h(x')},\label{eq:w.nh}\\
&=\alpha(x')\norm{\nabla \mathfrak{H}(x')}\norm{\nabla h(x')}^2+\alpha(x')\norm{\nabla h(x')}\inpr{\hat{v}_{\mathfrak{F}}(x')+\nabla\mathfrak{H}(x')}{\nabla h(x')},\notag\\
&=\alpha(x')\norm{\nabla \mathfrak{H}(x')}\norm{\nabla h(x')}^2+\alpha(x')\norm{\nabla h(x')}\inpr{\nabla\mathfrak{H}(x')}{\nabla h(x')},\notag\\
&\hspace{4.3cm}+\alpha(x')\norm{\nabla h(x')}\inpr{\hat{v}_{\mathfrak{F}}(x')}{\nabla h(x')},\notag\\
&=\alpha(x')\norm{\nabla \mathfrak{H}(x')}\norm{\nabla h(x')}^2+\alpha(x')\norm{\nabla\mathfrak{H}(x')}\norm{\nabla h(x')}^2\cos\vartheta,\notag\\
&\hspace{4.3cm}+\alpha(x')\norm{\nabla h(x')}\inpr{\hat{v}_{\mathfrak{F}}(x')}{\nabla h(x')},\notag\\
&=\alpha(x')\norm{\nabla \mathfrak{H}(x')}\norm{\nabla h(x')}^2(1+\cos\vartheta)+\alpha(x')\norm{\nabla h(x')}\inpr{\hat{v}_{\mathfrak{F}}(x')}{\nabla h(x')},\label{eq:w.nh.s}
\end{align}
where $\vartheta$ is the angle between $\nabla\mathfrak{H}(x')$ and $\nabla h(x')$. 

By an analogous computation, by \eqref{eq:the.vf} and \eqref{eq:forma} we have
\begin{align}
\inpr{\tilde{w}&(x')}{\nabla \mathfrak{H}(x')}=\norm{\nabla \mathfrak{H}(x')}\inpr{w_f(x')}{\nabla \mathfrak{H}(x')}+\alpha(x)\norm{\nabla h(x')}\norm{\nabla \mathfrak{H}(x')}^2,\label{eq:w.nH}\\
&=\alpha(x')\norm{\nabla h(x')}\norm{\nabla \mathfrak{H}(x')}^2(1+\cos\vartheta)+\norm{\nabla \mathfrak{H}(x')}\inpr{\hat{v}_f(x')}{\nabla \mathfrak{H}(x')},\label{eq:w.nH.s}
\end{align}
where, again, $\vartheta$ is the angle between $\nabla\mathfrak{H}(x')$ and $\nabla h(x')$.

Let us analyse these inner products for the different cases considered in Lemma~\ref{lem:uni.con.str.1}.
\paragraph{\textbf{Case 1}} \textit{Let $x  \in \mathring{\B}_{\e}^n \setminus V$ be such that the vectors $\nabla h(x)$ and $\nabla\mathfrak{H}(x)$ are colinear}.
In this case, in \eqref{eq:forma} and \eqref{eq:form}, both $\hat{v}_f(x)$ and $\hat{v}_{\mathfrak{F}}(x)$ are orthogonal to $\nabla\mathfrak{H}(x)$ and $\nabla h(x)$.
Hence \eqref{eq:w.nh.s} and \eqref{eq:w.nH.s} become
\begin{align*}
\inpr{\tilde{w}(x)}{\nabla h(x)}&=\alpha(x)\norm{\nabla \mathfrak{H}(x)}\norm{\nabla h(x)}^2(1+\cos\vartheta)>0,\\
\inpr{\tilde{w}(x)}{\nabla \mathfrak{H}(x)}&=\alpha(x)\norm{\nabla h(x)}\norm{\nabla \mathfrak{H}(x)}^2(1+\cos\vartheta)>0,
\end{align*}
since $\cos\vartheta>-1$, otherwise, the vectors $\nabla\mathfrak{H}(x)$ and $\nabla h(x)$ would point in opposite directions which is impossible by Proposition~3.5.
Thus, there exist an open neighbourhood $U_x$ of $x$ where these inner products are positive.

\paragraph{\textbf{Case 2}} \textit{The vectors $\nabla h(x)$ and $\nabla\mathfrak{H}(x)$ are linearly independent.} 
In this case the tube 
$N_x=\B_{\e}^n \cap f^{-1}(\s_{\norm{f(x)}}^{k-1})$ and the sphere $S_x=\s_{\norm{x}}^{n-1}$,  are transverse at $x$. 
Hence, its intersection $N_x\cap S_x$ is a submanifold of codimension $2$ of $U\subset\R^n$. 

We have two subcases.

\paragraph{\textbf{Subcase A}} \textit{The vector $\nabla\mathfrak{H}(x)$ is not in  $\langle P_x,\nabla h(x)\rangle$}. Then $N_x\cap S_x$ is transverse to the fibre
$f^{-1}(f(x))$, {\it i.e.},  $T_x(N_x\cap S_x)$ is transverse to  $T_xf^{-1}(f(x))$. By the construction of $w_f(x)$ and $w_{\mathfrak{F}}(x)$ in Lemma~~\ref{lem:uni.con.str.1} in this subcase, we also have that in \eqref{eq:forma} and \eqref{eq:form}, both $\hat{v}_f(x)$ and $\hat{v}_{\mathfrak{F}}(x)$ are orthogonal to $\nabla\mathfrak{H}(x')$ and $\nabla h(x')$. Hence, we get the same conclusion as in \textbf{Case 1}.

\paragraph{\textbf{Subcase B}} \textit{The vector $\nabla\mathfrak{H}(x)$ is in $\langle P_x,\nabla h(x)\rangle$}. Then $x\in M(f)$.
In this case $w_f(x)$ and $w_\mathfrak{F}(x)$ are collinear and by Proposition~\ref{prop:keyprop}, if $\e$ is small enough, they point in the same direction, that is, in \eqref{eq:mu.defi} $\mu(x)>0$.

By \eqref{eq:w.nh}, \eqref{eq:pi.wF.h}, \eqref{eq:w.nH} and \eqref{eq:pi.wf.H} we have
\begin{align*}
\inpr{\tilde{w}(x)}{\nabla h(x)}&=\alpha(x)\norm{\nabla \mathfrak{H}(x)}\norm{\nabla h(x)}^2+\alpha(x)^2\mu(x)\norm{\nabla h(x)}^3>0,\\
\inpr{\tilde{w}(x)}{\nabla \mathfrak{H}(x)}&=\frac{1}{\mu(x)}\norm{\nabla \mathfrak{H}(x)}^3+\alpha(x)\norm{\nabla h(x)}\norm{\nabla \mathfrak{H}(x)}^2>0.
\end{align*}

Hence $\inpr{\tilde{w}(x)}{\nabla h(x)}>0$ and $\inpr{\tilde{w}(x)}{\nabla \mathfrak{H}(x)}>0$ in all the cases, which implies that $\tilde{w}$ is transverse to the tubes
and to the spheres, satisfying properties \eqref{it:pr3} and \eqref{it:pr1}. Once we have the vector field $\tilde{w}$ defined in each neighbourhood $U_x$ using a partition of unity
we get a global vector field $\tilde{w}$ on $\mathring{\B}_{\e} \setminus V$ with the desired properties.
\end{proof}

\begin{rem}
Remember that the vector fields $w_f$ and $w_\mathfrak{F}$ have, respectively, the form given in \eqref{eq:forma} and \eqref{eq:form}. We can compute the precise values of $\alpha$ and $\beta$ as follows.
We have that
\begin{equation*}
 Df_x(w_f(x))=\nabla g(f(x)),
\end{equation*}
taking the inner product with $\nabla g(f(x))$
\begin{align}
\inpr{\nabla g(f(x))}{Df_x(w_f(x))}&=\norm{\nabla g(f(x))}^2 &&\notag\\
\inpr{\nabla h(x)}{w_f(x)}&=\norm{2f(x)}^2 &&\text{by the chain rule,}\notag\\
\alpha(x)\norm{\nabla h(x)}^2&=4\norm{f(x)}^2 &&\text{by \eqref{eq:forma}}\notag\\
\alpha(x)&=\frac{4\norm{f(x)}^2}{\norm{\nabla h(x)}^2}&&
\end{align}

Analogously, we have that
\begin{equation*}
D\mathfrak{F}_x(w_\mathfrak{F}(x))=\nabla g(\mathfrak{F}(x)),
\end{equation*}
taking the inner product with $\nabla g(\mathfrak{F}(x))$
\begin{align}
\inpr{\nabla g(\mathfrak{F}(x))}{D_\mathfrak{F}(w_\mathfrak{F}(x))}&=\inpr{\nabla g(\mathfrak{F}(x))}{\nabla g(\mathfrak{F}(x))}& &\notag\\
\inpr{\nabla\mathfrak{H}(x)}{w_\mathfrak{F}(x)}&=\norm{\nabla g(\mathfrak{F}(x))}^2 & &\text{by the chain rule,}\notag\\
\beta(x)\norm{\nabla\mathfrak{H}(x)}^2&=\norm{2\mathfrak{F}(x)}^2& &\text{by \eqref{eq:form}}\notag\\
\beta(x)4\norm{x}^2&=4\norm{x}^2&&\notag\\
\beta(x)&=1. &&\label{eq:beta}
\end{align}
\end{rem}

\begin{rem}\label{rem:mu.value}
Let $x\in M(f)$, then we have that $w_f(x)$ and $w_\mathfrak{F}(x)$ are collinear, that is  $w_\mathfrak{F}(x)=\mu(x)w_f(x)$, with $\mu(x)\neq0$ (see \eqref{eq:mu.defi}). 
Using the differential of the spherification map given in Proposition~\ref{prop:DF} we can compute the value of $\mu(x)$. Since $w_f(x)\in T_x E_{\ell_x}$ by Proposition~\ref{prop:sp.rp} we have
\begin{equation*}
D\mathfrak{F}_x(w_f(x))=\frac{\inpr{x}{w_f(x)}}{\norm{x}^2}\mathfrak{F}(x)=\frac{\inpr{2x}{w_f(x)}}{\norm{2x}^2}2\mathfrak{F}(x)=\frac{\inpr{\nabla\mathfrak{H}(x)}{w_f(x)}}{\norm{\nabla\mathfrak{H}(x)}^2}2\mathfrak{F}(x),
\end{equation*}
thus,
\begin{equation*}
D\mathfrak{F}_x\left(\frac{\norm{\nabla\mathfrak{H}(x)}^2}{\inpr{\nabla\mathfrak{H}(x)}{w_f(x)}}w_f(x)\right)=2\mathfrak{F}(x)=\nabla g(\mathfrak{F}(x)),
\end{equation*}
and we have that
\begin{equation}\label{eq:mu.val}
\mu(x)=\frac{\norm{\nabla\mathfrak{H}(x)}^2}{\inpr{\nabla\mathfrak{H}(x)}{w_f(x)}}.
\end{equation}
On the other hand, by Corollary~\ref{cor:v.El} we have that
\begin{align*}
Df_x(w_\mathfrak{F}(x))&=\frac{f(x)}{\norm{f(x)}^2}\inpr{f(x)}{Df_x(w_\mathfrak{F}(x))}=\frac{2f(x)}{4\norm{f(x)}^2}\inpr{2f(x)}{Df_x(w_\mathfrak{F}(x))}\\
     &=\frac{\inpr{\nabla g(f(x))}{Df_x(w_\mathfrak{F}(x))}}{4\norm{f(x)}^2}\nabla g(f(x))=\frac{\inpr{\nabla h(x)}{w_\mathfrak{F}(x)}}{4\norm{f(x)}^2}\nabla g(f(x))\\
\end{align*}
thus,
\begin{equation*}
Df_x\left(\frac{4\norm{f(x)}^2}{\inpr{\nabla h(x)}{w_\mathfrak{F}(x)}}w_\mathfrak{F}(x)\right)=\nabla g(f(x))
\end{equation*}
hence by \eqref{eq:mu.inve} we have that
\begin{equation}\label{eq:mu.inv.val}
\frac{1}{\mu(x)}=\frac{4\norm{f(x)}^2}{\inpr{\nabla h(x)}{w_\mathfrak{F}(x)}}.
\end{equation}
From \eqref{eq:mu.val} and \eqref{eq:mu.inv.val} we have that
\begin{equation}
\inpr{\nabla\mathfrak{H}(x)}{w_f(x)}\inpr{\nabla h(x)}{w_\mathfrak{F}(x)}=4\norm{f(x)}^2\norm{\nabla\mathfrak{H}(x)}^2,
\end{equation}
which implies that for any $x\in M(f)$ the inner products $\inpr{\nabla\mathfrak{H}(x)}{w_f(x)}$ and $\inpr{\nabla h(x)}{w_\mathfrak{F}(x)}$ always have the same sign, as we also can see from
equations \eqref{eq:pi.wF.h} and \eqref{eq:pi.wf.H}. By Proposition~\ref{prop:keyprop}, if $\e$ is small enough, these inner products are always positive.
\end{rem}

\begin{rem}
Let us revise what we have done. In Lemma~ \ref{lem:uni.con.str.1} we gave an open cover $\mathcal{U}=\{U_\lambda\}_{\lambda\in\Lambda}$ of $\B_\e^n\setminus V$ and in each open set $U_\lambda$ we constructed
the vector field $w_{f,\lambda}$ satisfying \eqref{it:wf}, the vector field $w_{\mathfrak{F},\lambda}$ satisfying \eqref{it:wF} and they also satisfy \eqref{it:wf.wF.nod}.
Let $\{\rho_\lambda\}_{\lambda\in\Lambda}$ be a partition of unity subordinated to the open cover $\mathcal{U}$. In Theorem~\ref{lem:uni.con.str.0} we constructed in each open set $U_\lambda$ the vector field
(using \eqref{eq:beta})
\begin{equation}
\tilde{w}_\lambda(x)=\Vert \nabla\mathfrak{H}(x)\Vert w_{f,\lambda}(x)+\alpha(x)\Vert \nabla h(x)\Vert w_{\mathfrak{F},\lambda}(x).
\end{equation}
which satisfies properties \eqref{it:tr.sph}, \eqref{it:tr.tube} and \eqref{it:cons.arg} and using the partition of unity we get the following global vector field with the same properties
\begin{align}
\tilde{w}(x)&=\sum_{\lambda\in\Lambda}\rho_\lambda(x) \tilde{w}_\lambda(x)
=\sum_{\lambda\in\Lambda}\rho_\lambda(x)\bigl(\Vert \nabla\mathfrak{H}(x)\Vert w_{f,\lambda}(x)+\alpha(x)\Vert \nabla h(x)\Vert w_{\mathfrak{F},\lambda}(x)\bigr),\notag\\
&=\Vert \nabla\mathfrak{H}(x)\Vert\sum_{\lambda\in\Lambda}\rho_\lambda(x) w_{f,\lambda}(x)+\alpha(x)\Vert \nabla h(x)\Vert\sum_{\lambda\in\Lambda}\rho_\lambda(x) w_{\mathfrak{F},\lambda}(x).\label{eq:global}
\end{align}
Is it easy to check that the vector fields $\sum_{\lambda\in\Lambda}\rho_\lambda(x) w_{f,\lambda}(x)$ and $\sum_{\lambda\in\Lambda}\rho_\lambda(x) w_{\mathfrak{F},\lambda}(x)$ are,
respectively, liftings to $\B_\e^n\setminus V$, by $f$ and $\mathfrak{F}$ of the radial vector field $\nabla g$, and their weighted sum \eqref{eq:global} gives the vector field $\tilde{w}$, 
as it is mentioned in the proof of \cite[Lemma~5.2]{Cisneros-Seade-Snoussi:d-regular}. In other words, it turns out that with the proof of Theorem~\ref{lem:uni.con.str.0}, 
the (incomplete) proof of \cite[Lemma~5.2]{Cisneros-Seade-Snoussi:d-regular} can be seen as a sketch of a proof.
\end{rem}

\begin{rem}\label{rem:gen.const}
Notice that the construction of the vector fields $w_f$ and $w_\mathfrak{F}$ with properties \eqref{it:wf}, \eqref{it:wF} and \eqref{it:wf.wF.nod} works in any open set $\mathcal{U}$ of 
$\R^n$ where $f$ and its spherefication $\mathfrak{F}$ are submersions, and defining the vector field $\tilde{w}$ as in \eqref{eq:the.vf} in the proof of Theorem~\ref{lem:uni.con.str.0} we get
a vector field on $\mathcal{U}$ satisfying properties \eqref{it:pr1}, \eqref{it:pr2} and \eqref{it:pr3}. By Proposition~\ref{prop:equiv} the spherefication $\mathfrak{F}$ is a submersion if
and only if the map $f/\|f\|$ is a submersion, thus, the construction of the vector field $\tilde{w}$ works on any open set $\mathcal{U}$ of $\R^n$ where $f$ and $f/\|f\|$ are submersions,
as it was mentioned in the Introduction.
\end{rem}

\section{Equivalence of the fibration on tube and on the sphere}\label{sec:equiv}

In this section, using Theorem~\ref{lem:uni.con.str.0}, we give for completeness the proof of the equivalence between the fibration on the tube and the fibration on the sphere for a real analytic  map
with isolated critical value.

Let $U$ be an open neighbourhood of $\0 \in \R^n$, $n >p$, and $f\colon(U,\0) \to (\R^p,0)$ a  non-constant analytic map defined
on $U$ with a critical point at $\0\in\R^n$ and $0\in\R^p$ is an isolated critical value. 
Also assume that $f$ is locally surjective (see Remark \ref{rem:surjective}).
Let $\B_\e^n$ be a closed ball in $\R^n$, centred at $\0$, of sufficiently small radius $\e$, so that every sphere in this ball, centred at $\0$, 
meets transversely every stratum of a Whitney stratification of $V$, and such that Proposition~\ref{prop:NOD} holds.

We say that $f$ satisfies the \textit{transversality property} if there exists $0<\delta \ll \e$ such that for every $y\in \B_\delta^p$ 
the fibre $f^{-1}(y)$ meets $\s_\e^{n-1}$ transversely. Since we are assuming $f$ locally surjective we can take $\B_\delta^p\subset\Ima(f|_{\B_{\e}})$.

\begin{rem}
If we work with a germ $f\colon(U,\0) \to (\R^p,0)$ instead of with just a map, we need the following stronger \textit{transversality property}.
We say that the germ $f$ satisfies the \textit{transversality property} if for any representative $f$ there exists a ball $\B_{\e_0}^n$ as above 
such that for every $0<\e<\e_0$ there exists a real number $\delta=\delta(\e)$ such that $0<\delta\ll\e$
and for every $y\in\B_\delta$ the  fibre $f^{-1}(y)$ meets $\s_{\e}^{n-1}$ transversely.
\end{rem}

\begin{rem}
In the case that $f$ is not locally surjective, for the transversality property we need to ask that every $y\in\B_\delta\cap\Ima(f|_{\B_{\e}})$ the  fibre $f^{-1}(y)$ meets $\s_{\e}^{n-1}$ transversely.
But if we consider that an empty fibre intersects transversely the sphere $\s_{\e}^{n-1}$ we can state in general the transversality property as above.
\end{rem}

\begin{prop}\label{prop:FTubo}
If $f$ satisfies the transversality property, then the restriction of $f$ to the Milnor tube
\begin{equation}\label{eq:tubo}
 f\colon N(\e,\delta) \rightarrow \s_\delta^{p-1},
\end{equation}
is a smooth fibre bundle.
\end{prop}

\begin{proof}
Since $f$ satisfies the transversality property we can apply Ehresmann Fibration Theorem for Manifolds with Boundary (see for instance \cite[p.~23]{Lamotke:TCPVASL}).
\end{proof}

Consider the projection $\pi\colon \B_\delta^p\setminus\{0\}\to\s^{p-1}$ given by $\pi(y)=\frac{y}{\norm{y}}$. We have that $\pi$ is a trivial fibre bundle over $\s^{p-1}$.
Let $\tilde{f}$ be the composition of the fibre bundle \eqref{eq:tubo} with $\pi$, since the restriction of $\pi$ to $\s_\delta^{p-1}$ is a diffeomorphism, we get an equivalent fibre bundle
\begin{equation}\label{eq:t.tubo}
\tilde{f}\colon N(\e,\delta) \rightarrow \s^{p-1}.
\end{equation}
Notice that $\tilde{f}$ is the restriction of the map $\Phi$ given in \eqref{eq:def.g} to the Milnor tube $N(\e,\delta)$.

\begin{thm}\label{thm:eq.icp}
Suppose $f\colon(U,\0) \to (\R^p,0)$ satisfies the transversality property. Then $f$ is $d$-regular if and only if the map
\begin{equation}\label{eq:esfera}
\phi:=\frac{f}{\norm{f}}\colon \s_\e^{n-1} \setminus K \to \s^{p-1}
\end{equation}
is a smooth fibre bundle which is equivalent to the fibre bundle \eqref{eq:t.tubo}.
\end{thm}

\begin{proof}
Since $f$ has the transversality property, by Proposition~\ref{prop:FTubo} we have the fibration \eqref{eq:t.tubo} on the tube. If $f$ is $d$-regular, by Theorem~\ref{lem:uni.con.str.0}
there exists a vector field $\tilde{w}$ satisfying properties \eqref{it:tr.sph}, \eqref{it:tr.tube} and \eqref{it:cons.arg}. The flow associated to $\tilde{w}$ defines
a diffeomorphism $\tau$ between the tube $N(\e,\delta)$ and $\s_\e^{n-1}\setminus f^{-1}(\mathring{\B}_\delta)$, where $\mathring{\B}_\delta$ is the open ball in $\R^p$ centred at $0$ of radius $\delta$,
``inflating'' the tube to the sphere in the following way: for a point $x\in N(\e,\delta)$ follow the integral curve of $\tilde{w}$ which passes through $x$ 
until it reaches $\s_\e^{n-1}$ at some point $\hat{x}$, which exists and it is unique because $\tilde{w}$ satisfies properties \eqref{it:tr.sph} and \eqref{it:tr.tube}. Define $\tau(x)=\hat{x}$. 
By property \eqref{it:cons.arg} the integral curves of $\tilde{w}$
lie on the $E_\ell$, so we have that $\tilde{f}(x)=\Phi(x)=\Phi(\hat{x})=\phi(\hat{x})$. Hence the diffeomorphism $\tau$ gives an equivalence of fibre bundles $\tilde{f}=\phi\circ\tau$
\begin{equation*}
\begin{tikzcd}[ampersand replacement=\&]
N(\e,\delta)\arrow[rr,"\tau"]\arrow[rd,"\tilde{f}"'] \& \& \s_\e\setminus f^{-1}(\mathring{\B}_\delta)\arrow[ld,"\phi"]\\
\& \s^{p-1} \&
\end{tikzcd} 
\end{equation*}
Now we extend the fibre bundle 
\begin{equation}\label{eq:tubo.inflado}
 \phi\colon\s_\e\setminus f^{-1}(\mathring{\B}_\delta)\to\s^{p-1}
\end{equation}
to $\s_\e\setminus K$ as follows. Since $\s_\e\cap f^{-1}(\B_\delta^p)$ is compact, the restriction of $f$
\begin{equation*}
f\colon \s_\e\cap f^{-1}(\B_\delta^p\setminus\{0\})\to \B_\delta^p\setminus\{0\}
\end{equation*}
is a proper submersion, and by Ehresmann Fibration Theorem, it is a smooth fibre bundle. Composing it with $\pi$, we get a fibre bundle given by the restriction of $\phi$
\begin{equation}\label{eq:tapa}
\phi|=\pi\circ f\colon \s_\e\cap f^{-1}(\B_\delta^p\setminus\{0\})\to\s^{p-1},
\end{equation}
since the composition of two smooth fibre bundles is again a smooth fibre bundle (see \cite[Lemma~A.1]{Cisneros-etal:FTDRDMGNICV}, \cite[Corollary~7]{Mckay:MCC} or  \cite{Ekedahl:CBPBP}).
We can glue the fibre bundles \eqref{eq:tubo.inflado} and \eqref{eq:tapa} since they coincide in the intersection and both have projection $\phi$, to get the smooth fibre bundle \eqref{eq:esfera}.

On the other hand, if \eqref{eq:esfera} is a smooth fibre bundle, then it is a submersion, and by Proposition~\ref{prop:equiv} $f$ is $d$-regular.
\end{proof}

\section{Real analytic maps with linear discriminant}\label{sec:lin.disc}

In this section, we see that Theorem~\ref{lem:uni.con.str.0} also holds for analytic maps with linear discriminant and we use it to 
prove the equivalence of the fibration on the tube and the fibration on the sphere for real analytic maps with arbitrary linear discriminant (Theorem~\ref{thm:eq.fib.ald} below). 
This is done in Subsection~\ref{ssec:ld}. 

\subsection{Equivalence of fibrations for maps with arbitrary linear discriminant}\label{ssec:ld}

We start recalling some definitions and results from \cite{Cisneros-etal:FTDRDMGNICV}. 
Let $f\colon (\R^n,0) \to (\R^p,0)$ be an analytic map with a critical point at $0$. Assume that $f$ is locally surjective (see Remark \ref{rem:surjective}). 
Equip $\R^n$ with a Whitney stratification adapted to $V=f^{-1}(0)$, and let $\B_{\e_0}^n$ be a closed  ball in $\R^n$, centred at $\0$, of sufficiently small radius $\e_0>0$, so that every sphere
in this ball, centred at $\0$, meets transversely every stratum of $V$.
In what follows, for $0<\e<\e_0$ we shall consider the restriction $f_{\e}$ of $f$ to the closed ball $\B_{\e}^n\subset\B_{\e_0}^n$.
Denote by $\Sigma_{\e}$ the critical set of $f_\e$ and by $\Delta_\e := f_\e(\Sigma_{\e})$ the \textit{discriminant} of $f_{\e}$.

\begin{rem}\label{rem:nice.germ}
\begin{itemize}
\item[$(i)$] The discriminant $\Delta_\e$ is a subanalitic set 
and it may depend on the choice of the radius $\e$, as showed in \cite{dosSantos-etal:FHSMG}. 
\item[$(ii)$] The results of this article extend to map-germs $f\colon (\R^n,0) \to (\R^p,0)$ if we consider $f$ in the class of \textit{nice analytic map-germs} defined in \cite[Definition~2.2]{dosSantos-etal:FHSMG}
for which the discriminant is a well-defined set-germ at $0$, so it does not depend on the radius $\e$.
\end{itemize}
\end{rem}

Let $0<\e<\e_0$, we say that the ball $\B_\e^n$ has the transversality property if there exist $0<\delta\ll\e$ such that for every $y\in\B_\delta^p\setminus\Delta_\e$ 
the fibre $f^{-1}(y)$ is transverse to the sphere $\s_\e$.
Since we are assuming $f$ locally surjective we can take $\B_\delta^p\subset\Ima(f_\e)$.

\begin{rem}
If we work with a nice analytic map-germ $f\colon(U,\0) \to (\R^p,0)$ (see Remark~\ref{rem:nice.germ}), the discriminant does not depend on the radius $\e$ and we just denote it by $\Delta$. 
In this case we need the following stronger \textit{transversality property}.
We say that $f$ satisfies the \textit{transversality property} if for every $0<\e<\e_0$ 
there exists a real number $\delta=\delta(\e)$ such that $0<\delta\ll\e$ and for every $y\in\B_\delta\setminus\Delta$ the fibre $f^{-1}(y)$ meets $\s_{\e}^{n-1}$ transversely.
\end{rem}

\begin{rem}
In the case that $f$ is not locally surjective for the transversality property we need to ask that every $y\in(\B_\delta^p\setminus\Delta_\e)\cap\Ima(f_\e)$ 
the fibre $f^{-1}(y)$ meets $\s_{\e}^{n-1}$ transversely.
But if we consider that an empty fibre intersects transversely the sphere $\s_{\e}^{n-1}$ we can state in general the transversality property as above.
\end{rem}

Suppose that $f$ has the transversality property, then Proposition~\ref{prop:FTubo} generalizes so that the restriction
\begin{equation}
 f_\e|\colon \B_\e^n \cap f^{-1}(\s_\delta^{p-1} \setminus \Delta_\e) \to \s_\delta^{p-1} \setminus \Delta_\e \label{eq:MT}
\end{equation}
is a smooth fibre bundle \cite[Theorem~2.7]{Cisneros-etal:FTDRDMGNICV}.

Let $0<\e<\e_0$, we say that $f$ has \textit{linear discriminant in the ball $\mathbb{B}_\e^n$}  if $\Delta_\e$ is a union of line-segments with one endpoint at $0 \in \mathbb{R}^p$. 
We say that $\eta>0$ is a \textit{linearity radius} for $\Delta_\e$ if each of these line-segments intersect $\s_\eta^{p-1}$, that is, if
\begin{align*}
\Delta_\e \cap \B_\eta^p = \Cone \big( \Delta_\e \cap \s_\eta^{p-1} \big) \, .
\end{align*}
The case when $f$ has $0\in\R^p$ as isolated critical value is considered to have linear discriminant with $\Delta_\e \cap \s_\eta^{p-1}=\emptyset$.

Let $f\colon (\B_{\e_0}^n,0) \to (\R^p,0)$ be an analytic map with linear discriminant in the ball $\mathbb{B}_\e^n$ and consider a linearity radius $\eta>0$ for $f$. Set:
\begin{equation*}
\mathcal{A}_{\eta} := \Delta_\e \cap \s_\eta^p.
\end{equation*}
Let $\pi_\eta\colon \s_{\eta}^p\to\s^{p-1}$ be the radial projection onto the unit sphere $\s^{p-1}$ and set $\mathcal{A}=\pi_\eta(\mathcal{A}_\eta)$.
For each point $\vartheta \in \s_\eta^{p-1}$, let $\mathcal{L}_\vartheta \subset \R^p$ be the open ray in $\R^p$ from the origin that contains the point $\vartheta$.
Set:
\begin{equation*}
 E_{\vartheta} := f_\e^{-1}(\mathcal{L}_\vartheta).
\end{equation*}
Notice that $E_{\vartheta}$ is a smooth manifold in $\B_\e^n$, for any $\vartheta$ in $\s_\eta^{p-1} \setminus \mathcal{A}_\eta$. 

We say that $f$ is \textit{$d$-regular (in the ball $\B_\e^n$)} if $E_{\vartheta}$ intersects the sphere $\s_{\e'}^{n-1}$ transversely in $\R^n$, for every $\e'$ with $0<\e' \leq \e$ and for every 
$\vartheta \in \s_\eta^{p-1} \setminus \mathcal{A}_\eta $.

\begin{rem}
In the case that  $f$ has an isolated critical value at $0\in\R^p$, $f$ is $d$-regular (in the definition given in Section~\ref{sec:d-reg}) if and only if there exists $0<\e<\e_0$ such that
$f$ is $d$-regular in the ball $\B_\e^n$.
\end{rem}

\begin{rem}\label{rem:delta=eta}
Let $f\colon (\B_{\e_0}^n,0) \to (\R^p,0)$ with $n\geq p\geq 2$ be a locally surjective analytic map.
Let $0<\e<\e_0$ and suppose the ball $\mathbb{B}_\e^n$ has the transversality property and $f$ has linear discriminant in $\mathbb{B}_\e^n$.
Notice that in this case we can take the linearity radius to be the $\delta$ in the definition of the transversality property.
\end{rem}

The following examples are also from \cite{Cisneros-etal:FTDRDMGNICV}.
\begin{examp}
The real analytic map $f\colon (\R^4,0) \to (\R^3,0)$
\begin{equation*}
f(x,y,z,w) := (x^2-y^2z, y, w) \, .
\end{equation*}
has the plane $\{x=y=0\}$ in $\R^4$ as its critical set, its discriminant is the axis $\{u_1 = u_2 =0\}$ in $\R^3$, so it is linear. One can check that $f$ is \textit{not} $d$-regular. 
\end{examp}

\begin{examp}
Let $\K$ be either $\R$ or $\C$. 
Let $(f,g): \K^n \to \K^2$ be a $\K$-analytic map of the form:
\begin{equation*}
(f,g) = \left( \sum_{i=1}^n a_i x_i^p \, , \, \sum_{i=1}^n b_i x_i^p \right) \, ,
\end{equation*}
where $p \geq 2$ is an integer and $a_i, b_i \in \K$ are constants in generic position, i.~e., the origin is in the convex hull of the points $(a_i,b_i)$ (which guarantees that the link of $V(f)$ is non-empty) and 
no two of the points $(a_i,b_i)$ are linearly dependent, that is, $a_i b_j \neq a_j b_i$, for any $i \neq j$ ({\it Weak Hyperbolicity Hypothesis}).
Its critical set $\Sigma$ is given by the coordinate axis of $\K^n$, the discriminant $\Delta$ is linear and it is $d$-regular (see \cite[Examples~3.6 and 3.7]{Cisneros-etal:FTDRDMGNICV}).
For $p=2$ this maps were studied by S.~L\'opez de Medrano in \cite{LopezdeMedrano:SHQM}.
\end{examp}

Let $0<\e<\e_0$ be such that $f$ has linear discriminant and it is $d$-regular in the ball $\B_\e^n$.
Consider the maps $\Phi\colon \B_\e^n\setminus f^{-1}(\Delta_\e)\to \s^{p-1} \setminus \mathcal{A}$ and 
$\mathfrak{F}\colon \B_\e^n\setminus f^{-1}(\Delta_\e)\to\R^p\setminus\Delta_\e$ given in
\eqref{eq:def.g} but restricted to $\B_\e^n\setminus f^{-1}(\Delta_\e)$.
Hence Proposition~\ref{prop:equiv} generalizes in a straightforward way for maps with linear discriminant, removing $f^{-1}(\Delta_\e)$ from the domain 
of the maps involved instead of only removing $V$ \cite[Proposition~3.8]{Cisneros-etal:FTDRDMGNICV}.
Thus, the maps $\mathfrak{F}$ and $\phi=\Phi|\colon  \s_\e^{n-1} \setminus f^{-1}(\Delta_\e) \to \s^{p-1} \setminus \mathcal{A}$ are submersions if and only if $f$ is $d$-regular. 
In fact, we have the following theorem.

\begin{thm}[{\cite[Theorem~3.9]{Cisneros-etal:FTDRDMGNICV}}]\label{prop_2}
Let $f\colon (\B_{\e_0}^n,0) \to (\R^p,0)$ with $n\geq p\geq 2$ be a locally surjective analytic map.
Let $0<\e<\e_0$ and suppose the ball $\mathbb{B}_\e^n$ has the transversality property and $f$ has linear discriminant in $\mathbb{B}_\e^n$.
If $f$ is $d$-regular then the restriction of $\Phi$ given by
\begin{equation}\label{eq:FS}
 \phi=\Phi|\colon  \s_\e^{n-1} \setminus f^{-1}(\Delta) \to \s^{p-1} \setminus \mathcal{A}
\end{equation}
is a smooth locally trivial fibration.
\end{thm}

The construction of the vector field $\tilde{w}$ of Theorem \ref{lem:uni.con.str.0} works in any open set of $\R^n$ where $f$ and the spherefication $\mathfrak{F}$ are
submersions (see Remark~\ref{rem:gen.const}), so in Subsection~\ref{ssec:construction} we can replace $\mathring{\B}_{\e}^n\setminus V$ by $\mathring{\B}_{\e}^n\setminus f^{-1}(\Delta_\e)$. 
Using this generalization of Theorem~\ref{lem:uni.con.str.0} we 
prove the equivalence of fibrations \eqref{eq:MT} and \eqref{eq:FS} for real analytic maps with arbitrary linear discriminant as in Theorem~\ref{thm:eq.icp}.

\begin{thm}\label{thm:eq.fib.ald}
Let $f\colon (\B_{\e_0}^n,0) \to (\R^p,0)$ with $n\geq p\geq 2$ be a locally surjective analytic map.
Let $0<\e<\e_0$ and suppose the ball $\mathbb{B}_\e^n$ has the transversality property, $f$ has linear discriminant and it is $d$-regular in $\mathbb{B}_\e^n$.
Then the fibre bundles:
\begin{gather*}
\tilde{f}:=\pi_\delta\circ f_\e|\colon \B_\e^n \cap f^{-1}(\s_\delta^{p-1} \setminus \mathcal{A}_\delta) \to \s^{p-1} \setminus \mathcal{A} \\
\intertext{and}
\phi\colon \s_\e^{n-1} \setminus f^{-1}(\Delta) \to \s^{p-1} \setminus \mathcal{A}
\end{gather*}
are equivalent, where $\pi_\delta\colon \s_{\delta}^{p-1} \to \s^{p-1}$ is the radial projection.
\end{thm}

\begin{rem}
In Theorem~\ref{thm:eq.fib.ald} we take the linearity radius to be $\delta$ as in Remark~\ref{rem:delta=eta}.
\end{rem}

\begin{rem}
In the PhD thesis \cite{Ribeiro} and in the articles \cite{dosSantos-etal:FHSMG,dosSantos-etal:MHSFEP} the authors study nice real analytic maps-germs $f\colon (\R^n,0) \to (\R^p,0)$ 
with discriminant of positive dimension obtaining some results in common with \cite{Cisneros-etal:FTDRDMGNICV}.
They  prove the existence of the fibration on the tube \eqref{eq:MT} if $f$ satisfies a condition equivalent to the transversality property \cite[Lemma~3.3]{dosSantos-etal:FHSMG}.
They also consider real analytic map-germs with linear discriminant (they call it with radial discriminant) and if they have fibration on the tube and are $d$-regular (they call it $\rho$-regular)
then the fibration on the sphere \eqref{eq:FS} exists. Then they consider the problem of finding conditions for which the two fibrations are equivalent \cite[Problem~4.1]{dosSantos-etal:MHSFEP}
and prove that the problem reduces to find a vector field which satisfies properties \eqref{it:pr1}, \eqref{it:pr2} and \eqref{it:pr3} 
(which they call a \textit{good vector field} or \textit{Milnor vector field}).
The authors conjecture that if $f$ is a real analytic map which is $d$-regular and for which both fibrations \eqref{eq:MT} and \eqref{eq:FS} exists, 
then they are equivalent \cite[Equivalence Conjecture~4.4]{dosSantos-etal:MHSFEP}. Theorem~\ref{thm:eq.fib.ald} above proves that the conjecture is true.
\end{rem}

\begin{rem}
In \cite{Ribeiro,dosSantos-etal:MHSFEP} the authors 
prove that for fixed $\vartheta \in \s_\eta^{p-1} \setminus \mathcal{A}_\eta$ 
the corresponding fibres of fibrations \eqref{eq:MT} and \eqref{eq:FS} are diffeomorphic but without being able to prove the equivalence of fibrations since they could not find a 
radius $\e_1$ which works for \textit{every} $\vartheta\in\s_\eta^{p-1} \setminus \mathcal{A}_\eta$.
\end{rem}

\section*{Acknowledgments}

The first author was supported by CONACYT~253506 and a UNAM-DGAPA-PASPA sabbatical scholarship. He wants to thank Prof.~Hans Brodersen for the helpful email correspondence.
Both authors want to thank José Seade and Jawad Snoussi for many enlightening conversations.


\begin{thebibliography}{10}

\bibitem{DosSantos:UMCSMF}
R.~N. {Ara{\'u}jo dos Santos}.
\newblock {Uniform (m)-condition and strong {M}ilnor fibrations}.
\newblock In {\em {Singularities {II}}}, volume 475 of {\em {Contemp. Math.}},
  pages 189--198. Amer. Math. Soc., Providence, RI, 2008.

\bibitem{dosSantos:ERMFQHS}
R.~N. {Ara{\'u}jo dos Santos}.
\newblock {Equivalence of real {M}ilnor fibrations for quasi-homogeneous
  singularities}.
\newblock {\em Rocky Mountain J. Math.}, 42(2):439--449, 2012.

\bibitem{dosSantos-etal:SOBSFRM}
Raimundo~N. {Ara{\'u}jo dos Santos}, Ying Chen, and Mihai Tib{\u a}r.
\newblock {Singular open book structures from real mappings}.
\newblock {\em Cent. Eur. J. Math.}, 11(5):817--828, 2013.

\bibitem{dosSantos-etal:FHSMG}
Raimundo~N. {Ara{\'u}jo dos Santos}, Maico~F. Ribeiro, and Mihai Tib{\u a}r.
\newblock {Fibrations of highly singular map germs}.
\newblock {\em Bull. Sci. Math.}, 155:92--111, 2019.

\bibitem{dosSantos-etal:MHSFEP}
Raimundo~N. {Ara{\'u}jo dos Santos}, Maico~F. Ribeiro, and Mihai Tib{\u a}r.
\newblock {Milnor-{H}amm sphere fibrations and the equivalence problem}.
\newblock {\em J. Math. Soc. Japan}, 72(3):945--957, 2020.

\bibitem{Burghelea-Verona:LHPAS}
Dan Burghelea and Andrei Verona.
\newblock {Local homological properties of analytic sets}.
\newblock {\em Manuscripta Math.}, 7:55--66, 1972.

\bibitem{Cisneros-Menegon:EMFMLF}
Jos{\'e}~Luis Cisneros-Molina and Aur{\'e}lio Menegon.
\newblock Equivalence of {M}ilnor and {M}ilnor-{L}{\^e} fibrations for real
  analytic maps.
\newblock {\em Internat. J. Math.}, 30(14):1950078, 25, 2019.

\bibitem{Cisneros-etal:FTDRDMGNICV}
Jos{\'e}~Luis Cisneros-Molina, Aur{\'e}lio Menegon, Jos{\'e} Seade, and Jawad
  Snoussi.
\newblock Fibration theorems {\`a} la {M}ilnor for differentiable maps with
  non-isolated singularities, February 2020.
\newblock Preprint : arXiv:2002.07120 [math.AG].

\bibitem{Cisneros-Seade-Snoussi:d-regular}
Jos{\'e}~Luis Cisneros-Molina, Jos{\'e} Seade, and Jawad Snoussi.
\newblock Milnor fibrations and {$d$}-regularity for real analytic
  singularities.
\newblock {\em Internat. J. Math.}, 21(4):419--434, 2010.

\bibitem{dosSantos-Tibar:RMGHOBS}
R.~N.~Ara{\'u}jo dos Santos and M.~Tib{\u a}r.
\newblock {Real map germs and higher open book structures}.
\newblock {\em Geom. Dedicata}, 147:177--185, 2010.

\bibitem{HB14}
{Hans Brodersen}.
\newblock About $d$-regularity.
\newblock Personal comunication, September 2014.

\bibitem{Ekedahl:CBPBP}
\href{http://mathoverflow.net/users/4008/torsten- ekedahl}{Torsten Ekedahl}.
\newblock Is the composition of two bundle projections necessarily a bundle
  projection?
\newblock MathOverflow.
\newblock
  URL:\href{http://mathoverflow.net/q/74722}{http://mathoverflow.net/q/74722}.

\bibitem{Jacquemard:These}
Alain Jacquemard.
\newblock {\em Th{\`e}se 3{\`e}me cycle}.
\newblock PhD thesis, Universit{\'e} de Dijon, 1982.

\bibitem{Jacquemard:FMPAR}
Alain Jacquemard.
\newblock {Fibrations de {M}ilnor pour des applications r{\'e}elles}.
\newblock {\em Boll. Un. Mat. Ital. B (7)}, 3(3):591--600, 1989.

\bibitem{Lamotke:TCPVASL}
Klaus Lamotke.
\newblock {The topology of complex projective varieties after {S}.
  {L}efschetz}.
\newblock {\em Topology}, 20(1):15--51, 1981.

\bibitem{Le1}
Dung~Tr{\'a}ng L{\^e}.
\newblock {Some remarks on relative monodromy}.
\newblock In {P. Holm}, editor, {\em {Real and complex singularities (Proc.
  Ninth Nordic Summer School/NAVF Sympos. Math., Oslo, 1976)}}, pages 397--403.
  Sijthoff and Noordhoff, Alphen aan den Rijn, 1977.

\bibitem{Looijenga:ISPCI}
E.~J.~{\null N}. Looijenga.
\newblock {\em Isolated singular points on complete intersections}, volume~77
  of {\em London Mathematical Society Lecture Note Series}.
\newblock Cambridge University Press, Cambridge, 1984.

\bibitem{LopezdeMedrano:SHQM}
Santiago {L{\'o}pez de Medrano}.
\newblock {Singularities of homogeneous quadratic mappings}.
\newblock {\em Rev. R. Acad. Cienc. Exactas F{\'\i} s. Nat. Ser. A Math.
  RACSAM}, 108(1):95--112, 2014.

\bibitem{Massey:RAMFSLI}
David~B. Massey.
\newblock {Real analytic {M}ilnor fibrations and a strong {\L}ojasiewicz
  inequality}.
\newblock In {\em {Real and complex singularities}}, volume 380 of {\em {London
  Math. Soc. Lecture Note Ser.}}, pages 268--292. Cambridge Univ. Press,
  Cambridge, 2010.

\bibitem{Mckay:MCC}
Benjamin McKay.
\newblock {Morphisms of {C}artan {C}onnections}.
\newblock arXiv:0802.1473, 2010.

\bibitem{Milnor:ISH}
John Milnor.
\newblock {On isolated singularities of hypersurfaces}.
\newblock Preprint, June 1966.

\bibitem{Milnor:SPCH}
John Milnor.
\newblock {\em Singular points of complex hypersurfaces}.
\newblock Annals of Mathematics Studies, No. 61. Princeton University Press,
  Princeton, N.J., 1968.

\bibitem{Hansen:MFTRS}
{Nikolai B. Hansen}.
\newblock Milnor's {F}ibration {T}heorem for {R}eal {S}ingularities.
\newblock Master's thesis, Faculty of Mathematics and Natural Sciences
  University of Oslo, May 2014.

\bibitem{Pichon-Seade:FMAfgb}
Anne Pichon and Jos{\'e} Seade.
\newblock {Fibred multilinks and singularities {$f\overline g$}}.
\newblock {\em Math. Ann.}, 342(3):487--514, 2008.

\bibitem{Ribeiro}
Maico Felipe~Silva Ribeiro.
\newblock {\em Singular Milnor Fibrations}.
\newblock PhD thesis, Universidade de S\~ao Paulo, April 2018.

\bibitem{RS}
Maria Aparecida~Soares Ruas and Raimundo Nonato~Ara{\'u}jo dos Santos.
\newblock {Real {M}ilnor fibrations and (c)-regularity}.
\newblock {\em Manuscripta Math.}, 117(2):207--218, 2005.

\bibitem{Ruas-Seade-Verjovsky:RSMF}
Maria Aparecida~Soares Ruas, Jos{\'e} Seade, and Alberto Verjovsky.
\newblock {On real singularities with a {M}ilnor fibration}.
\newblock In A.~Libgober and M.~Tibar, editors, {\em {Trends in
  singularities}}, {Trends Math.}, pages 191--213. Birkh{\"a}user, Basel, 2002.

\bibitem{dosSantos-Ribeiro:GCEMVF}
Raimundo Ara{\'u}jo~Dos Santos and Maico~F. Ribeiro.
\newblock {Geometrical conditions for the existence of a Milnor vector field},
  October 2018.

\bibitem{Seade:OBDAHVF}
Jos{\'e} Seade.
\newblock {Open book decompositions associated to holomorphic vector fields}.
\newblock {\em Bol. Soc. Mat. Mexicana (3)}, 3(2):323--335, 1997.

\bibitem{Seade:OMFTIOA50Y}
Jos{\'e}~A. Seade.
\newblock {On Milnor's fibration theorem and its offspring after 50 years}.
\newblock {\em Bull. Amer. Math. Soc.}, November 2018.

\bibitem{Verdier:SWTBS}
Jean-Louis Verdier.
\newblock {Stratifications de {W}hitney et th{\'e}or{\`e}me de
  {B}ertini-{S}ard}.
\newblock {\em Invent. Math.}, 36:295--312, 1976.

\end{thebibliography}
\end{document}